\documentclass[12pt]{amsart}
\usepackage{amssymb, latexsym, euscript}
\usepackage{amssymb}
\usepackage{latexsym}
\usepackage{amsmath}

\def\sD{{\mathfrak D}}      
   \def\sH{{\mathfrak H}}   
   \def\sK{{\mathfrak K}}   \def\sL{{\mathfrak L}}
\def\sM{{\mathfrak M}}   \def\sN{{\mathfrak N}}

      \def\dC{{\mathbb C}}

      \def\dR{{\mathbb R}}

\def\cA{{\mathcal A}}      \def\cC{{\mathcal C}}
      
\def\cG{{\mathcal G}}   \def\cH{{\mathcal H}}   
      \def\cL{{\mathcal L}}
\def\cM{{\mathcal M}}   \def\cN{{\mathcal N}}   
\def\cP{{\mathcal P}}      
\def\cS{{\mathcal S}}      \def\cU{{\mathcal U}}
      
\def\cY{{\mathcal Y}}   

   \def\bB{{\mathbf B}}

\def\dim{{\rm dim\,}}
\def\ran{{\rm ran\,}}
\def\cran{{\rm \overline{ran}\,}}
\def\dom{{\rm dom\,}}

\def\cdom{{\rm \overline{dom}\,}}

\def\codim{{\rm codim\,}}

\def\uphar{{\upharpoonright\,}}

\def\f{\varphi}

\def\half{{\frac{1}{2}}}

\newtheorem{theorem}{Theorem}[section]
\newtheorem{lemma}[theorem]{Lemma}
\newtheorem{proposition}[theorem]{Proposition}
\newtheorem{corollary}[theorem]{Corollary}
\newtheorem{definition}[theorem]{Definition}
\newtheorem{remark}[theorem]{Remark}

\numberwithin{equation}{section}

\def\RE{{\rm Re\,}}
\def\IM{{\rm Im\,}}
\def\wt{\widetilde}
\def\wh{\widehat}
\def\ti{\widetilde}

\usepackage{color}

\begin{document}
\title
[Compressions of extensions]
{Compressions of selfadjoint and maximal dissipative extensions of non-densely defined symmetric operators}

\author[Yu.M.~Arlinski\u{\i}]{Yu.M. Arlinski\u{\i}}
\address {Volodymyr Dahl East Ukrainian National University, Kyiv,
 Ukraine}

\email{yury.arlinskii@gmail.com}


\dedicatory {Dedicated to the memory of Heinz Langer}
\begin{abstract}
Selfadjoint and maximal dissipative extensions of a non-densely defined symmetric operator $S$ in an infinite-dimensional separable Hilbert space are considered and their compressions on the subspace ${\rm \overline{dom}\,} S$ are studied. The main focus is on the case ${\rm codim\,}{\rm \overline{dom}\,}S=\infty$. 
 New properties of the characteristic functions of non-densely defined symmetric operators are
established.

\end{abstract}
\maketitle
\thispagestyle{empty}
\section{Introduction}
{\textit{Notations}}.
We use the following notations. The Banach space of all bounded operators acting between Hilbert spaces $H_1$ and $H_2$ is denoted by $\bB(H_1,H_2)$ and $\bB(H):=\bB(H,H)$.
The symbols $\dom T$, $\ran T$, $\ker T$ denote
the domain, range and kernel of a linear operator $T$, respectively,
$\cdom T$ and $\cran T$ denote the closure of the domain and the range of $T$.
\textit{Throughout the paper the word "subspace" means a closed linear manifold}.
If $\sL$ is a subspace, the orthogonal
projection onto $\sL$ is denoted by $P_\sL.$ The identity operator in a Hilbert space $H$ is denoted by $I_H$ and sometimes by $I$.
By $\cL^\perp$ we denote the orthogonal complement to the linear manifold $\cL$.
The notation $T\uphar \cN$ means the restriction of a linear operator $T$ to the linear manifold $\cN\subset\dom T$.

The open upper/lower half-plane of the complex plane $\dC$ are denoted by $\dC_\pm \!:=\! \{ z\!\in\!\dC: \IM z \!\gtrless\! 0\}$ and
$\dR_+:=(0,+\infty).$
If $Z$ is a contraction, then $D_Z:=(I-Z^*Z)^\half$ and $\sD_Z:=\cran D_Z.$
If $B$ is a linear operator and $\cL$ is a linear manifold, then
$
B^{-1}\{\cL\}:=\{h\in\dom B: Bh\in\cL\}.
$
\vskip 0.3 cm

Let $\cH$ be an infinite-dimensional complex separable Hilbert space,
 let $\cH_0$ be a subspace (closed linear manifold) of $\cH$ and let $P_{\cH_0}$ be the orthogonal projection in $\cH$ onto $\cH_0$. If $T$ is a linear operator in $\cH$, then the operator
$$T_0:=P_{\cH_0}T\uphar(\dom T\cap\cH_0)$$
is called the compression of $T$. If $T$ is unbounded, then it is possible that $\dom T\cap\cH_0=\{0\}$ but if $\dom T$ is dense in $\cH$ and if $\codim\cH_0<\infty$, i.e.,
$\dim (\cH\ominus\cH_0)<\infty$, then (see \cite[Lemma 2.1]{GoKr}) the linear manifold $\dom T \cap\cH_0$ is dense in $\cH_0.$
Moreover, it is established by Stenger \cite{Stenger} that if the operator $T$ is selfadjoint in $\cH$ and $\codim \cH_0< \infty$, then the compression $T_0$ is a selfadjoint operator in $\cH_0$.

Later Stenger's result has been generalized by Nudelman  \cite[Theorem 2.2]{Nud} for the case of any maximal dissipative operator.
Recall that a linear operator $T$ in a Hilbert space $\cH$ is called \textit{dissipative} if $\IM (Tf,f)\ge 0$ $\forall f\in\dom T$ and \textit{maximal dissipative} if it is dissipative and has no dissipative extensions without exit from $\cH$.
In \cite{ADW, ACD}  Stenger's and Nudelman's statements have been extended for selfadjoint and maximal dissipative linear relations in Hilbert and Kre\u{\i}n spaces.

Dijksma and Langer \cite{DL2017, DL2018, DL2020, DL2020b} and  Mogilevski\u{\i} \cite{Mog2019, Mog2020} studied compressions of selfadjoint exit space extensions via the Kre\u{\i}n formula for generalized (compressed) resolvents \cite{Krein1944,Krein1946,KrL1971,DM,BHS2020}.

The main goal of the present paper is to study compressions of selfadjoint and maximal dissipative extensions of a non-densely defined symmetric operators via von-Neumann's \cite{AG}, Na\u{\i}mark's \cite{Naimark3}, Krasnosel'ski\u{\i}'s \cite{krasno} and Shtraus' \cite{Straus1968} formulas. More precisely, let
$S$ be a non-densely defined closed symmetric operator
 in an infinite-dimensional complex separable Hilbert space $\cH$ and let $\wt S$ be a maximal dissipative or a selfadjoint (if the deficiency indices are equal) extension of $S$ in $\cH$. Set $\cH_0:=\cdom S$. We are interested in the compressions
$$\wt S_0:=P_{\cH_0}\wt S\uphar(\dom \wt S\cap\cH_0),\; \wt S_{*0}:=P_{\cH_0}\wt S^*\uphar(\dom \wt S\cap\cH_0).$$
If  $S_0:=P_{\cH_0}S$, then $S_0$ is a symmetric and densely defined operator in $\cH_0$, $\wt S_0\supseteq S_0$ and $\wt S_{*0}\supseteq S_0$.
Clearly, if $\wt S$ is a selfadjoint extension of $S$, then $\wt S_0=\wt S_{*0}$ is a symmetric (selfadjoint when $\codim \cH_0<\infty$) in $\cH_0$  and if $\wt S$ is maximal dissipative, then $\wt S_0$ and $-\wt S_{*0}$ are dissipative (maximal dissipative when $\codim \cH_0<\infty$) in $\cH_0$ and, moreover, $\left<\wt S_0,\wt S_{*0}\right>$  is a \textit{dual pair}:
$$(\wt S_0 f, g)=(f,\wt S_{*0}g)\;\;\forall f\in\dom\wt S_0,\;\forall g\in\dom\wt S_{*0}.$$
We additionally suppose that \textit{the operator $S_0$ is closed} (such a symmetric operator $S$ is called \textit{regular} \cite{Shmulyan70, ArlBelTsek2011}). In the case $\codim \cH_0<\infty$ it is established in \cite{Shmulyan70} that $S$ is regular, for abstract examples of regular symmetric operators  in the case $\codim \cH_0<\infty$ see e.g. \cite[Proposition 2.2]{Arl_arxiv2024}. When the selfadjoint (maximal dissipative) extension $\wt S$ of a regular symmetric operator possesses the property that \textit{$P_{\cH_0}\wt S$ is closed} (such extensions are called \textit{regular} \cite{ArlBelTsek2011}), it is established in \cite[Theorem 2.5.7, Theorem 4.1.10]{ArlBelTsek2011} that also in case $\codim \cH_0=\infty$ the compression $\wt S_0$ is a selfadjoint (maximal dissipative) extension of $S_0$ in $\cH_0$.

Let  $\sN_\lambda$ ($\lambda\in\dC\setminus\dR)$ be the deficiency subspace of a symmetric $S$. Set $\sL_\lambda:=P_{\sN_\lambda}\cH_0^\perp.$ Then the operator $V_{\lambda}:P_{\sN_{\lambda}}f\mapsto
P_{\sN_{\bar \lambda}}f, $ $f\in\cH_0^\perp $ is isometric \cite{krasno}. It is proved in \cite{Shmulyan70} that $S$ is regular iff $\sL_\lambda$ is closed at least for one non-real $\lambda$.

There is a one-to-one correspondence between symmetric (dissipative) extensions of a closed non-densely defined symmetric operators $S$ and admissible isometric (contractive if $\IM \lambda>0$) operators  $\sN_{\lambda}\supseteq\dom U\stackrel{U}\longrightarrow\sN_{\bar\lambda}$ given by the formulas
\begin{equation}\label{feb7}
\left\{\begin{array}{l}
\dom \wt S=\dom S\dot+(I-U)\dom U,\\
\wt S(f_S+\f_\lambda-U\f_\lambda)=Sf_S+\lambda\f_\lambda-\bar\lambda U\f_\lambda,\; \; f_S\in\dom S,\;\f_\lambda\in\dom U.
\end{array}\right.
\end{equation}
The admissibility means (see \cite[Theorem 1]{Naimark3}, \cite{krasno, Straus1968}) that
$(U-V_\lambda)f\ne 0 \;\;\forall f\in\dom U\cap(\sL_\lambda\setminus\{0\}).$
Moreover, the operator $\wt S$ in \eqref{feb7} is selfadjoint (maximal dissipative) if and only if $U$ is an admissible unitary acting from $\sN_\lambda$ onto $\sN_{\bar\lambda}$
(admissible contraction with $\dom U=\sN_\lambda$).

Set $\sN_\lambda':=\sN_\lambda\ominus\sL_\lambda$. Then $\sN'_\lambda\subset\cH_0$ and $\sN'_\lambda=\cH_0\ominus(S_0-\bar\lambda I_{\cH_0})\dom S_0$ \cite{krasno}, i.e., $\sN'_\lambda$ is the deficiency subspace of densely defined in $\cH_0$ symmetric operator $S_0$. The subspaces $\sN'_\lambda$ are called the \textit{semi-deficiency subspaces} of $S$ \cite{krasno}.

Further for concreteness we will use $\lambda=i$. It is shown that if
a contractive admissible operator $U\in\bB(\sN_i,\sN_{-i})$ is represented as the block-operator matrix
\[
U=
\begin{bmatrix} U_{11} &U_{12} \cr
U_{21}& U_{22}\end{bmatrix}:\begin{array}{l}\sL_{i}\\ \oplus\\
\sN'_i\end{array}\to \begin{array}{l}\sL_{-i}\\ \oplus\\
\sN'_{-i}\end{array},
\]
then $\dom\wt S_0=\dom S_0\dot +(I-U_0)\dom U_0$, $\dom\wt S_{*0}=\dom S_0\dot +(I-U_{*0})\dom U_{*0}$, where $U_0$ and $U_{*0}$ are the following contractions
\[
\begin{array}{l}
 \dom U_0=\left\{\f_i'\in\sN_i':U_{12}\f_{i}'\in\ran (U_{11}-V_i)\right\},\;U_0\f_i'=(U_{22}-U_{21}(U_{11}-V_i)^{-1}U_{12})\f_i',\\[2mm]
\dom U_{*0}=\left\{\f_{-i}'\in\sN_{-i}':U_{21}^*\f_{-i}'\in\ran (U_{11}^*-V_{-i})\right\},\;U_{*0}\f_{-i}'=U_{22}^*-U_{12}^*(U_{11}^*-V_{-i})^{-1}U_{21}^*\f_{-i}'.
\end{array}
\]
Moreover,
\begin{itemize}
\item the pair $\left< U_0, U_{*0}\right>$ forms a dual pair, i.e., the equality
\[
(U_0f,g)=(f, U_{*0}g)\;\;\forall f\in\dom U_,\;\forall g\in\dom U_{*0}
\]
holds, 
\item if $U$ is unitary, then $U_0$ and $U_{*0}$ are isometries and $U_{*0}=U^{-1}_0$ and the selfadjointness of $\wt S_0$ is equivalent to the inclusions
\begin{equation}\label{apr14a}
\ran U_{12}\subset\ran (U_{11}-V_i)\;\&\;\ran U_{21}^*\subset\ran (U_{11}^*-V_{-i});
\end{equation}
\item the following are equivalent:
\begin{enumerate}
\def\labelenumi{\rm (\roman{enumi})}
\item both operators $\wt S_0$ and $-\wt S_{*0}$ are maximal dissipative in $\cH_0$,
\item $\wt S_{*0}=(\wt S_0)^*$,
\item $\dom U_0=\sN_i'$,\; $\dom U_{*0}=\sN_{-i}'$
\item the inclusions in \eqref{apr14a}
hold;
\end{enumerate}
\item if $\dim(\dom S)^\perp<\infty$, then $\dom U_0=\sN'_i$ and $\dom U_{*0}=\sN'_{-i}$, i.e., $\wt S_0$ and $-\wt S_{*0}$ are maximal dissipative operators in $\cH_0$ and $(\wt S_0)^*=\wt S_{*0}$ in $\cH_0$.
\end{itemize}
In the case $\dim(\dom S)^\perp=\infty$ we prove that
\begin{enumerate}
\item
for any closed symmetric $\wt S_0$ in $\cH_0$ such that $\wt S_0\supseteq S_0$ there exist selfadjoint extensions $\wt S$ of $S$ in $\cH$ whose compressions on $\cH_0$  coincide with $\wt S_0$,
\item if a dual pair $\left<\wt S_0,\wt S_{*0}\right>$ of closed extensions of $S_0$ in $\cH_0$ is such that $\wt S_0$ and $-\wt S_{*0}$ are dissipative operators and the equalities
\[
\begin{array}{l}
\{f\in\dom \wt S_0:\IM(\wt S_0f, f)=0\} 
=
\{g\in\dom\wt S_{*0}:\IM(\wt S_{*0}g,g)=0\}\\[2mm]
=\dom\wt S_0\cap\dom \wt S_{*0} 
\end{array}
\]
 hold, then
there exist maximal dissipative extensions $\wt S$ of $S$ such that their compressions $P_{\cH_0}\wt S\uphar(\dom \wt S\cap\cH_0)$ and $P_{\cH_0}\wt S^*\uphar(\dom \wt S^*\cap\cH_0)$ coincide with $\wt S_0$ and $\wt S_{*0}$, respectively.
\end{enumerate}
Observe that the above equalities are valid if both $\wt S_0$ and $-\wt S_{*0}$ are maximal dissipative
(see \cite{Kuzhel},
\cite[Lemma 3.3]{ArlCAOT2023}).

Our proof depends mainly (see Proposition \ref{mar14a}, Theorem \ref{bef25a}, Theorem \ref{bef29ab}, Theorem \ref{mar11aaa}, Theorem \ref{juni18a}) on the properties of the contraction
$$Y:=P_{\sL_{-i}}UV_{-i}\uphar\sL_{-i}.$$
In terms of $Y$ the admissibility of a contraction $U\in\bB(\sN_i,\sN_{-i})$ is equivalent to the equality $\ker (I-Y)=\{0\}$, for admissible \textit{unitary} $U$ the selfadjointness of $\wt S_0$ (or the validity of inclusions in \eqref{apr14a}) is equivalent to the  equalities
 \begin{equation}\label{apr16a}
\ran (I-Y)=\ran (I-Y^*)=\ran\left(I-\half(Y+Y^*)\right)^\half,
\end{equation}
and the regularity of $\wt S$ holds if and only if
\begin{equation}\label{juni30aa}
\ran (I-Y)=\sL_{-i}.
\end{equation}
If $U\in\bB(\sN_i,\sN_{-i})$ is an admissible contraction, then 
$\wt S$ is a regular $\Longleftrightarrow$
 $\wt S^*$ is a regular $\Longleftrightarrow $
the equality in \eqref{juni30aa} holds.
Besides, the equalities in \eqref{apr16a} are sufficient for the equality $(\wt S_0)^*=\wt S_{*0}$.
Clearly, if  $\dim\sL_{-i}(=\dim (\dom S)^\perp)<\infty$, then \eqref{juni30aa} is valid.

 In our study we essentially use  parameterizations of $2\times 2$ block operator contractive and unitary matrices, see \cite{AG1982,DKW,FoFra1984,ShYa}.

We establish some new properties of the characteristic functions of a closed symmetric operator $S$ defined by A.V.~Shtraus \cite{Straus1968} in the form

\begin{equation}\label{kbdinh}
\dC_+\ni z\stackrel{C_\lambda^{S}(z)}\longrightarrow \bB(\sN_{\lambda},\sN_{\bar\lambda}),\; C_\lambda^{S}(z):=(\wt S_z-\lambda I)(\wt S_z-\bar\lambda I)^{-1}\uphar\sN_{\lambda},
\end{equation}
where the operator $\wt S_z$ is the maximal dissipative extension:
 \begin{equation}\label{slamb}
\dom \wt S_z=\dom S+\sN_z,\; \wt S_z(f_S+\f_z)=Sf_S+z\f_z,\; f_S\in\dom S,\;\f_z\in\sN_z.
\end{equation}
We call such operators the \textit{Shtraus extension}. From \eqref{kbdinh} and \eqref{slamb} it follows that
\[
\dom \wt S_z=\dom S\dot+(I-C_\lambda^{S}(z))\sN_\lambda.
\]
We show that for a non-densely defined symmetric operator the equality  $(\wt S_z)_0=(\wt S_0)_z $ is valid and the limit (non-tangential to the real axis)
\[
\lim\limits_{z\to\infty,z\in\dC_+}C_\lambda^{S}(z)f=V_{\lambda}f\;\;\forall f\in\sL_{\lambda}
\]
holds.
Moreover, using the results related to the compressions of maximal dissipative extensions (Section \ref{apr11c}), we show (see Proposition \ref{apr10bb}) that the characteristic function $C_\lambda^{S_0}(z):\sN_\lambda'\to\sN_{\bar\lambda}'$ of the operator $S_0$ is the Schur complement of the operator matrix
$$C_\lambda^{S}(z)-V_\lambda P_{\sL_{\lambda}} : \begin{array}{l}\sL_{\lambda}\\ \oplus\\
\sN'_\lambda\end{array}\to \begin{array}{l}\sL_{\bar\lambda}\\ \oplus\\
\sN'_{\bar\lambda}\end{array}.$$

Let $S$ be a densely defined closed symmetric operator in the Hilbert space $\cH$.  Assume that $\cH$ is a subspace of a Hilbert space $\sH$, $\cH_1:=\sH\ominus\cH,$
$ \dim\cH_1=\infty$. Then in $\sH$ the operator $S$ is non-densely defined and regular symmetric, the deficiency subspace $\sN_\lambda$ of $S$ in $\cH$ is the semi-deficiency subspace of $S$ in $\sH$, $\sL_\lambda= \cH_1$ and $V_\lambda=I_{\cH_1}$ for all $\lambda\in \dC\setminus\dR$. Therefore, results of Sections \ref{mar16}, \ref{apr11b}, \ref{apr11c} can be applied for selfadjoint and maximal dissipative extensions of $S$ in $\sH$ and their compressions on $\cH$.
 In particular, relying on Theorem \ref{bef29ab} and Theorem \ref{mar11aaa}, in Theorem \ref{aug06a} it is shown for symmetric $S$ with infinite deficiency indices the existence of a selfadjoint/ maximal dissipative extensions $\wt S$ in $\sH$ of the second kind \cite{Naimark3}, i.e., $\dom \wt S\cap\cH=\dom S$ (respectively, $\dom \wt S\cap\cH=\dom\wt S^*\cap\cH=\dom S$) and having additional property $\dom \wt S\cap\cH_1=\{0\}$ (respectively $\dom \wt S\cap\cH_1=\dom\wt S^*\cap\cH_1=\{0\}$).

\section{Symmetric and maximal dissipative operators}
\subsection{Symmetric operators}

Recall that a complex number $\lambda$ is a point of regular type of the linear operator $T$ \cite{AG} if there exists a positive number $k(\lambda)$ such that
\[
||(T-\lambda I)f||\ge k(\lambda)||f||\;\forall f\in\dom T.
\]
If $T$ is a closed linear operator and $\ker (T-\lambda I)=\{0\}$, then the linear manifold $\ran (T-\lambda I)$ is a subspace in $\sH$ if and only if the number $\lambda$ is a point of regular type of $T$. We denote by $\wh\rho(T)$ the set of all points of regular type (the field of regularity) of the operator $T$.

Recall that a linear operator $S$ in a Hilbert space $\cH$ is called symmetric (or Hermitian) if $\IM (Sf,f)=0$ for all $f\in\dom S$. If $\dom S$ is dense, then $S$ is symmetric iff $S\subseteq S^*$.

Let $S$ be a closed symmetric operator in $\cH$. Then the set $\dC\setminus\dR\subseteq\wh\rho(S)$.  
 Set
$$\sM_\lambda:=\ran (S-\lambda I),\;\sN_\lambda:=\sM_{\bar \lambda}^\perp.$$
The numbers $n_\pm=\dim \sN_\lambda,\; \lambda\in\dC_\pm $
are called the deficiency indices (defect numbers) of $S$ \cite{AG,krasno}.

If $S$ is densely defined, then due to J.~von Neumann's results \cite{AG}
\begin{enumerate}
\item
the domain of the adjoint operator $S^*$ for each $\lambda\in\dC\setminus\dR$ admits the direct decomposition
$\dom S^*=\dom S\dot+\sN_\lambda\dot+\sN_{\bar\lambda};$
\item $S$ admits selfadjoint extensions in $\cH$ if and only if the deficiency indices of $S$ are equal; moreover, fix $\lambda\in\dC\setminus\dR$, then the formulas in
\eqref{feb7}
give a one-to-one
correspondence between all unitary operators $U$, acting from $\sN_\lambda$ onto $\sN_{\bar\lambda}$ and all selfadjoint extensions of $S$.
\end{enumerate}
\subsubsection{Non-densely defined symmetric operators}
If $\cdom S\ne \cH$, then $\dom S\cap\sN_\lambda=\{0\}$.  
 Set
\begin{equation}\label{jjpyfx}
\cH_0:=\cdom S,\;\sL:=\cH\ominus\cH_0,\; S_0:=P_{\cH_0}S.
\end{equation}
It is proved in (\cite[Lemma 2]{krasno}, \cite[Theorem 7]{Naimark3}) that
\begin{equation}\label{KRASn}
\sL\cap\sM_{\lambda}=\{0\}\;\;\forall \lambda\in\dC\setminus\dR
\end{equation}
and (see \cite[Lemma 2, Theorem 8]{krasno})
\begin{equation}\label{xthdjy}
(S-\bar\lambda I)(S-\lambda I)^{-1}P_{\sM_\lambda}h=P_{\sM_{\bar\lambda}}h\quad\mbox{for all}\quad
h\in\sL,\; \lambda\in\dC\setminus\dR.
\end{equation}
The adjoint of $S$ is the linear relation (multi-valued operator). But if
the operator $S$ we consider as a closed operator acting from $\cH_0$ into $\cH$, then since $\dom S$ is dense in $\cH_0$, there exists the adjoint operator $S^*$ acting from $\cH$ into $\cH_0$, $S^*$ is closed and densely defined in $\cH$.
Clearly, $S^*_0=S^*\uphar(\cH_0\cap\dom S^*)$ and (see \cite{ArlBelTsek2011})
\begin{equation}\label{bef15a}
S^*\f_\lambda=P_{\cH_0}\f_\lambda,\; \f_\lambda\in\sN_\lambda.
\end{equation}
The domain $\dom S^*$ admits the decomposition $\dom S^*=\dom S+\sN_\lambda+\sN_{\bar\lambda}$ \cite[Theorem 1]{Shmulyan70} but $\dom S, $ $\sN_\lambda$, $\sN_{\bar\lambda}$ are not linearly independent \cite{krasno, Naimark3}.
Moreover,
\begin{equation}\label{gthtctx1}
\begin{array}{l}
\left\{\begin{array}{l}f_S+\f_\lambda+\f_{\bar\lambda}=0\\f_S\in\dom S,\;\f_\lambda\in\sN_\lambda,\;\f_{\bar\lambda}\in\sN_{\bar\lambda}
\end{array}\right.
\Longleftrightarrow
\left\{\begin{array}{l}\f_\lambda=P_{\sN_\lambda}h,\;
\f_{\bar\lambda}=-P_{\sN_{\bar\lambda}}h,\\
f_S=2i\IM\lambda(S-\lambda I)^{-1} P_{\sM_\lambda}h,\;
h\in\sL\end{array}\right..
\end{array}
\end{equation}
 Therefore
\begin{equation}\label{rhfc11}
\sN_{\bar\lambda}\cap(\dom S\dot+\sN_\lambda)=P_{\sN_{\bar\lambda}}\sL.
\end{equation}
Set
\begin{equation}\label{gjlghjcn}
\sL_\lambda:=P_{\sN_\lambda}\sL,\; \sN_\lambda':=\sN_\lambda\ominus\sL_\lambda
\end{equation}
 Then (see \cite{krasno}) the operator $V_\lambda:\sL_\lambda\mapsto\sL_{\bar\lambda}$ defined by the equality
\begin{equation}\label{forbis}
V_\lambda P_{\sN_\lambda}h=P_{\sN_{\bar \lambda}} h,\; h\in\sL,\;\lambda \in\dC\setminus\dR
\end{equation}
is an isometry. Clearly,
$ V_{\bar\lambda}=V^{-1}_\lambda=V^*_{\lambda}\uphar\sL_\lambda$ and it follows from \eqref{xthdjy} and \eqref{forbis} that
\[
\begin{array}{l}
(I-V_\lambda)P_{\sN_\lambda}h=P_{\sN_\lambda}h-P_{\sN_{\bar\lambda}}h=P_{\sM_\lambda}h-P_{\sM_{\bar\lambda}}h \\[2mm]
=P_{\sM_\lambda}h-(S-\bar\lambda I)(S-\lambda I)^{-1}P_{\sM_\lambda}h
=2i\IM\lambda (S-\lambda I)^{-1}P_{\sM_\lambda}h\in\dom S\;\;\forall h\in\sL.
\end{array}
\]
The subspace $\sN_\lambda'$ lies in $\cH_0$ and  is called the semi-deficiency subspace of $S$. It is the deficiency subspace of the densely defined in $\cH_0$ symmetric operator $S_0$ (see \eqref{jjpyfx}), i.e.,  $\ker (S^*-\lambda I)\cap \cH_0=\sN_\lambda'.$ and the numbers
$
n_{\pm}'=\dim\sN'_\lambda,\;\lambda\in\dC_\pm
$
are called the semi-deficiency indices of $S$.
It is proved in \cite[Corollary to Theorem 2]{krasno} that the linear manifolds $\sL_{\lambda}$ for $\lambda\in\dC\setminus\dR$ are either all closed (i.e. are subspaces) or none of them is closed. The theorem below is established in \cite[Theorem 3]{Shmulyan70} (see also \cite[Theorem 2.4.1]{ArlBelTsek2011}).
\begin{theorem}\label{ivekmzy}
 The following are equivalent:
 \def\labelenumi{\rm (\roman{enumi})}
\begin{enumerate}
\def\labelenumi{\rm (\roman{enumi})}
\item the operator $S_0$ is closed;
\item the linear manifold $\sL_\lambda$ is closed (i.e., $\sL_\lambda$ is a subspace) for at least one (then for all) non-real $\lambda$;
\end{enumerate}
\end{theorem}
Clearly, (1) if $\dim\sL<\infty$, then $S_0$ is closed, (2) the operator $S_0$ is selfadjoint operator in $\cH_0$ if and only if
$\sN_\lambda=\sL_\lambda=P_{\sN_\lambda}\sL\;\;\forall\lambda\in\dC\setminus\dR$
and
$
\dom S^*=\dom S\dot+\sN_\lambda.$

It is established in \cite[Lemma 3]{Naimark3} (see also \cite[Theorem 5]{krasno}) that von Neumann's result related to the existence of selfadjoint extensions of a non-densely defined symmetric operator $S$ holds true, i.e., $S$ admits selfadjoint extensions in $\cH$ if and only if the deficiency indices of $S$ are equal and the description of all selfadjoint extensions is of the form \eqref{feb7} with such unitary operator $U$ from $\sN_\lambda$ onto $\sN_{\bar\lambda}$ that
\[
(U-V_\lambda)\f_\lambda\ne 0\;\;\forall \f_\lambda\in\sL_\lambda\setminus\{0\},
\]
where $\sL_\lambda$ and $V_\lambda$ are defined in \eqref{gjlghjcn} and \eqref{forbis}, respectively.
\begin{definition} \label{bef13a} \cite{ArlBelTsek2011}
\begin{enumerate}
\item The operator $S$ is called regular if $S_0$ is closed.
\item An extension $\wt S$ of $S$ is called regular if the operator $P_{\cH_0}\wt S$ is closed.
\end{enumerate}
\end{definition}
From the equality $S^*\uphar\dom\wt S= P_{\cH_0}\wt S$ it follows that
a selfadjoint extension $\wt S$ of $S$ is regular if and only if the domain $\dom \wt S$ is a subspace in the Hilbert space
$\cH^+_{S^*}:$
\begin{equation}\label{aug12a}
\cH^+_{S^*}:=\dom S^*, \; (f,g)_{\cH^+_{S^*}}=(f,g)+(S^*f,S^*g),\; f,g\in\dom S^*.
\end{equation}

\subsection{Maximal dissipative extensions of symmetric operators and their Cayley transforms}
It is well known that  (see e.g. \cite{Ka,Kuzhel, {Phillips1959}, Straus1968})
\begin{itemize}
\item a densely defined dissipative operator $T$ is maximal dissipative  if and only if $-T^*$ is dissipative; 
\item the resolvent set $\rho(T)$ of a maximal dissipative operator $T$ contains the open lower half-plane and 
the resolvent admits the estimate
$||(T-\lambda I)^{-1}||\le(|\IM \lambda|)^{-1}\;\; \forall\lambda\in\dC_-$ ($\IM \lambda<0$).
\end{itemize}
If $T$ is a maximal dissipative operator in $\cH$, then its  Cayley transform
\begin{equation}\label{rtkbnh}
Y_\lambda=(T-\lambda I)(T-\bar \lambda I )^{-1}\Longleftrightarrow \left\{\begin{array}{l} \psi=(T-\bar \lambda I )f\\
Y_\lambda \psi=(T-\lambda I)f\end{array}\right. f\in\dom T,\; \lambda\in\dC_+
\end{equation}
is a contraction defined on $\cH$ (see \cite{AG}, \cite{Straus1968}). It follows from \eqref{rtkbnh} that
\[
I-Y_\lambda =2i\IM \lambda (T-\bar\lambda I)^{-1},\;\ran (I-Y_\lambda)=\dom T,
\]
\[
||D_{Y_\lambda}\psi||^2=4\IM \lambda\,\IM (Tf,f),\; \psi=(T-\bar\lambda I)f,\;f\in\dom T,
\]
\begin{equation}\label{cghzo}
Y_\lambda^*=(T^*-\bar\lambda I)(T^*-\lambda I)^{-1}\Longleftrightarrow
\left\{\begin{array}{l} \phi=(T^*-\lambda I )g\\
Y_\lambda^* \phi=(T^*-\bar \lambda I)g\end{array}\right., \; g\in\dom T^*.
\end{equation}
From \eqref{rtkbnh} and \eqref{cghzo} one gets
\begin{equation}\label{jhfnyjt}
\left\{\begin{array}{l}f=(I-Y_\lambda)h\\
Tf=(\lambda I-\bar\lambda Y_\lambda)h\end{array}\right.,\;  \left\{\begin{array}{l}g=(I-Y_\lambda^*)\psi\\
T^*g=(\bar \lambda I-\lambda Y_\lambda^*)\psi\end{array}\right.,\; h, \psi\in\cH.
\end{equation}
According to \cite[Theorem 1.1, Theorem 1.2]{Straus1968}
if
$S$ is a closed symmetric operator and $\lambda\in\dC_+$, then the formulas
\begin{equation}\label{ljghfci}
\left\{ \begin{array}{l}\dom T=\dom S\dot+(I-M)\sN_\lambda\\
T(f_S+(I-M)\f_\lambda)=Sf_S+\lambda\f_\lambda-\bar\lambda M\f_\lambda\qquad (f_S\in\dom S,\;\f_\lambda\in\sN_\lambda)
\end{array}\right.
\end{equation}
establish a one-to-one correspondence between all maximal dissipative extensions $T$ of $S$ and all contractions
$M\in\bB(\sN_\lambda,\sN_{\bar\lambda})$ such that
\[
(M-V_\lambda)\psi_\lambda\ne 0\;\; \forall \psi_\lambda\in\sL_\lambda\setminus\{0\},
\]
where $V_\lambda$ is defined by \eqref{forbis}.
The adjoint operator $T^*$ can be described as follows
\begin{equation}\label{ljghfci2}
\left\{ \begin{array}{l}\dom T^*=\dom S\dot+(I-M^*)\sN_{\bar\lambda}\\
T^*(f_S+(I-M ^*)\f_{\bar \lambda})=Sf_S+\bar\lambda\f_{\bar\lambda}-\lambda M ^*\f_{\bar\lambda}\qquad (f_S\in\dom S,\;\f_{\bar\lambda}\in\sN_{\bar\lambda})
\end{array}\right..
\end{equation}
Observe that if $T$ is a maximal dissipative extension of $S$ of the form \eqref{ljghfci}, then the Cayley transform \eqref{rtkbnh}
is the orthogonal sum
\[
Y_\lambda=U_\lambda \oplus M,\;U_\lambda=(S-\lambda I)(S-\bar\lambda I)^{-1}:\sM_{\bar\lambda}\to\sM,\; M:=Y_\lambda\uphar\sN_\lambda:\sN_\lambda\to\sN_{\bar\lambda}.
\]
Since $U_\lambda$ is an isometry, we get that
\begin{equation}\label{defoperl1}
D^2_{Y_\lambda}=D^2_{M}P_{\sN_\lambda},\;D^2_{Y^*_\lambda}=D^2_{M^*}P_{\sN_{\bar\lambda}}.
\end{equation}
If $T$ is a maximal dissipative extension of a densely defined symmetri8c operator $S$, then von Neumann's formula for $\dom S^*$, \eqref{ljghfci}, \eqref{ljghfci2},  and \eqref{defoperl1} yield for the Cayley transform \eqref{rtkbnh} of $T$ that
\[
\dom (I-Y_\lambda)\cap\sD_{Y^*_\lambda}=\dom (I-Y^*_\lambda)\cap\sD_{Y_\lambda}=\{0\}.
\]
\begin{remark}\label{mar9bb} Let $T$ be an extension of a regular non densely defined symmetric operator $S$ in $\cH$. Suppose that
$\rho(T)\ne\emptyset$ of $T$ and $T^*$ is an extension of $S$ as well.
Then for each $\lambda\in\dC\setminus\dR$ such that $\bar\lambda\in\rho(T)$ the operators $T$ and $T^*$ admit the representations \eqref{ljghfci} and \eqref{ljghfci2}, respectively, with $M=(T-\lambda I)(T-\bar\lambda I)^{-1}\uphar\sN_\lambda\in\bB(\sN_\lambda,\sN_{\bar\lambda}).$  Moreover, the following are equivalent:
(see \cite[Theorem 2.5.2, Theorem 4.1.5, Remark 4.1.6]{ArlBelTsek2011})
\begin{enumerate}
\def\labelenumi{\rm (\roman{enumi})}
\item the operator $P_{\cH_0}T$ is closed,
\item the operator $P_{\cH_0}T^*$ is closed,
\item the linear manifold $(M-V_\lambda)\sL_\lambda$ is a subspace,
\item the linear manifold $(M^*-V_{\bar\lambda})\sL_{\bar\lambda}$ is a subspace.
\end{enumerate}
\end{remark}
\section{The Shtraus extensions and characteristic functions of symmetric operators}\label{jul11}
Let $z\in \wh\rho(S)$. Consider the  Shtraus extension $\wt S_z$ of $S$ \cite{Straus1968} (see \eqref{slamb}):
\[
\dom \wt S_z=\dom S\dot+\sN_z,\; \wt S_z(f_S+\f_z)=Sf_S+z\f_z,\; f_S\in\dom S,\;\f_z\in\sN_z.
\]
The operator $\wt S_z$ is maximal dissipative for $ z\in\dC_+$, maximal accumulative for $z\in\dC_-$, and selfadjoint for $z\in\wh\rho(S)\cap\dR$ \cite{Straus1968}.
Moreover,
$\wt S^*_z=\wt S_{\bar z}.$
\subsection{The Cayley tranform of the Shtraus extensions}
\begin{proposition}\label{apr7a}Let $S$ be a closed symmetric operator.
Let $z\in\dC\setminus\dR$ and let
$$ Y_z=(\wt S_z-z I)(\wt S_z-\bar z I)^{-1}$$
be the Cayley transform of $\wt S_z$. Then
\begin{enumerate}
\item
$ Y_z$ is the partial isometry of the form
\[
 Y_z=U_zP_{\sM_{\bar z}},\; U_z=(S-z I)(S-\bar z I)^{-1},\;Y_z\uphar\sN_z=0,\; Y^*_z=Y_{\bar z},
\]
and
\begin{equation}\label{defoperl11}
\begin{array}{l}
D_{Y_z}=P_{\sN_z},\;\sD_{Y_z}=\sN_z,\; \dim \sD_{Y_z}=n_+\\[2mm]
D_{Y_z^*}=P_{\sN_{\bar z}},\;\sD_{ Y_z^*}=\sN_{\bar z},\; \dim\sD_{Y_z^*}=n_-.
\end{array}
\end{equation}
\item For densely defined  $S$ the relations
\begin{equation}\label{30jul}
\ran (I-Y_z)\cap\sD_{Y^*_z}=\ran (I-Y_z^*)\cap\sD_{Y_z}=\{0\}
\end{equation}
 hold.
\item If $S$ is non-densely defined, then
\begin{equation}\label{mar16bb}
\begin{array}{l}
\dom\wt S_z\cap\dom \wt S^*_z=\dom S\dot+\sL_z=\dom S\dot+\sL_{\bar z},\\[2mm]
\dom \wt S_z\cap\ran D_{Y_z^*}=\ran (I-Y_z)\cap\ran D_{Y^*_z}=\sL_{\bar z},\\[2mm]
 \dom \wt S^*_z\cap\ran D_{Y_z}=\ran (I- Y^*_z)\cap\ran D_{ Y_z}=\sL_{z}.
\end{array}
\end{equation}
Moreover,
\begin{itemize}
\item if $S_0$ is selfadjoint ($n_\pm '=0$), then 
$$\sD_{ Y^*_z}\subset\ran (I- Y_z)\quad \mbox{and}\quad \sD_{ Y_z}\subset\ran (I- Y^*_z);$$
\item if $n'_+=0$, $n'_->0$, then
 $$\ran D_{Y_z}=\sD_{ Y_z}\subset\ran (I-Y^*_z),\;\ran D_{Y_z^*} \cap\ran (I- Y_z)=\sL_{\bar z};  $$
 \item if $n'_-=0$, $n'_+>0$, then
 $$\ran D_{Y^*_z}=\sD_{ Y^*_z}\subset\ran (I- Y_z),\; \ran D_{ Y_z}\cap\ran (I-Y^*_z)=\sL_{z}.  $$
\end{itemize}
In addition
\begin{equation}\label{apr8a}
(\wt S_z-\bar z I)^{-1}h=-\cfrac{1}{2i\IM z}\, P_{\sN_{\bar z}}h,
(Y_z-I)h=-P_{\sN_{\bar z}}h\;\;\forall h\in\sL
\end{equation}
and
\begin{equation}\label{apr9a}
(I-Y_z)^{-1}P_{\sN_{\bar z}}h=h,\; P_{\sN_z}(I-Y_z)^{-1}P_{\sN_{\bar z}}h=P_{\sN_z}h=V_{\bar z}P_{\sN_{\bar z}}h\;\;\forall h\in\sL.
\end{equation}
\end{enumerate}
\end{proposition}
\begin{proof}
The equalities in \eqref{defoperl11} follow from \eqref{defoperl1}. Relations in \eqref{mar16bb}  are consequences of \eqref{gthtctx1} and \eqref{defoperl11}.
The von Neumann formula $\dom S^*=\dom S\dot +\sN_z\dot+\sN_{\bar z}$, equalities $\dom \wt S_z=\ran (I-Y_z)$, $\dom \wt S_{\bar z}=\ran (I-Y^*_z)$ and \eqref{defoperl11} imply \eqref{30jul}.

Let us proof \eqref{apr8a}. From the equality $\dom\wt S_z\cap\dom \wt S^*_z=\dom S\dot+\sL_{\bar z}$ it follows that for each $h\in \sL$ the vector
$P_{\sN_{\bar z}}h\in\sL_{\bar z}$ admits the unique  decomposition
\[
P_{\sN_{\bar z}}h=f_S+\f_z,\; f_S\in\dom S,\;\f_z\in\sN_{z}.
\]
Hence, from \eqref{gthtctx1} for $\lambda=\bar z$ we obtain that
\[
\f_z=-P_{\sN_z} h, \; f_S=-2i\IM z(S-\bar zI)^{-1}P_{\sM_{\bar z}}h.
\]
Therefore,
\[
(\wt S_z-\bar z I)P_{\sN_{\bar z}}h=(S-\bar zI)f_S+(z-\bar z)\f_z=-2i\IM z P_{\sM_{\bar z}}h-2i\IM zP_{\sN_z} h=-(2i\IM z) h.
\]
This and \eqref{rtkbnh} yield \eqref{apr8a}. Then \eqref{apr9a} follows from \eqref{apr8a}.
\end{proof}
\subsection{The characteristic function of a symmetric operator}

 Let $C^S_{\lambda}(z)$, $z,\lambda\in\dC_+$ be the characteristic functions  of the operator $S$ (see \eqref{kbdinh})
\[
\dC_+\ni z\to \bB(\sN_{\lambda},\sN_{\bar\lambda}):C_\lambda^{S}(z):=(\wt S_z-\lambda I)(\wt S_z-\bar\lambda I)^{-1}\uphar\sN_{\lambda},
\]
i.e.,
\begin{equation} \label{kbdinh12}
\dom\wt S_z=\dom S\dot+\sN_z=\dom S\dot+(I-C^S_{\lambda}(z))\sN_{\lambda}.
\end{equation}
Since $\sN_z\subset\dom S\dot+(I-C^S_{\lambda}(z))\sN_{\lambda}$, we have from \eqref{kbdinh12} that for each $\f\in\sN_z$ there exist unique $f_S\in\dom S$ and $\psi\in\sN_\lambda$ such that the equalities
\[
\left\{\begin{array}{l}\f=f_S+(I-\cC^S_\lambda(z))\psi\\
z\f=Sf_S+\lambda f-\bar\lambda\, \cC^S_\lambda (z) \psi
\end{array}\right.
\] 
hold. Consequently
\begin{equation}\label{pes11a}
\cC^S_\lambda (z)P_{\sN_\lambda}\f=\cfrac{z-\lambda}{z-\bar{\lambda}}\,P_{\sN_{\bar\lambda}}\f\;\;\forall \f\in\sN_z,\;\:z,\lambda\in\dC_+.
\end{equation}

It is proved in \cite{Straus1968} that
\begin{equation}\label{charfunc2}
\begin{array}{l}
C_\lambda^{S}(z)=\cfrac{z-\lambda}{2i\IM \lambda}P_{\sN_{\bar\lambda}}(\wt S_{\bar\lambda}-\lambda I)(\wt S_{\bar\lambda}-z I)^{-1}\uphar\sN_{\lambda}
=\cfrac{z-\lambda}{z-\bar\lambda}\,\cP_{z,\bar\lambda}\uphar\sN_{\lambda},\; z\in\dC_+,
\end{array}\end{equation}
where $\cP_{z,\bar\lambda}$ is a skew projector in $\cH$ onto $\sN_{\bar\lambda}$ corresponding to the direct decomposition $\cH=\sM_{z}\dot+\sN_{\bar\lambda},$
and
\begin{equation}\label{31jul}
||C_\lambda^{S}(z)||\le\left|\cfrac{z-\lambda}{z-\bar{\lambda}}\right|\;\; \forall z\in\dC_+.
\end{equation}

\begin{proposition}\label{charfunc11}
Let $\lambda\in\dC_+$. Then the characteristic function $C^{S}_{\lambda}(z)$ of the symmetric operator $S$ takes the form
\begin{equation} \label{snhf}
C^{S}_{\lambda}(z)=\frac{z-\lambda}{z-\bar\lambda}\,P_{\sN_{\bar\lambda}}\big (I- \frac{z-\lambda}{z-\bar\lambda}\, Y^*_{\lambda}\big)^{-1}\uphar\sN_{\lambda},
\;\forall z\in\dC_+
\end{equation}
where
$ Y_{\lambda}=(\wt S_{\lambda}-\lambda I)(\wt S_{\lambda}-\bar \lambda I)^{-1} =( S_{\lambda}-\lambda I)( S_{\lambda}-\bar \lambda I)^{-1}P_{\sM_{\bar\lambda}}
$
is the Cayley transform of $\wt S_{\lambda}.$

If $S$ is a non-densely defined closed symmetric operator, then
\begin{equation}\label{apr9bb}
\lim\limits_{z\to\infty,z\in\dC_+}C_\lambda^{S}(z)f=V_{\lambda}f\;\;\forall f\in\sL_{\lambda},
\end{equation}
where $\sL_{\lambda}$ and $V_{\lambda}$ are defined by \eqref{gjlghjcn} and \eqref{forbis}, respectively, and the limit is non-tangential to the real axis.
\end{proposition}
\begin{proof}
Since $\wt S_{\bar\lambda}=\wt S^*_{\bar\lambda}$ from \eqref{jhfnyjt} we get
\[\begin{array}{l}
(\wt S_{\bar{\lambda}}-\lambda I)=-2i\IM \lambda(I-Y^*_{\lambda}),\;
(\wt S_{\bar{\lambda}}-z I)=-(z-\bar\lambda)\left(I-\frac{z-\lambda}{z-\bar\lambda}Y^*_{\lambda}\right)\left(I-Y^*_{\lambda}\right)^{-1}.
\end{array}
\]
Now from \eqref{charfunc2} follows \eqref{snhf}.

Assume $S$ is a non-densely defined closed symmetric operator.
Because $Y_{\lambda}^*=Y_{\bar\lambda}$,
the equalities in \eqref{apr8a} and \eqref{apr9a} imply that $\sL_{\lambda}=P_{\sN_{\lambda}}(\dom S)^\perp\subset\ran (I-Y^*_{\lambda})$
and
$$(I-Y^*_{\lambda})^{-1}P_{\sN_{\lambda}}h=h\;\;\forall h\in  h\in\sL=(\dom S)^\perp .$$
Since $Y_{\lambda}^*$ is a contraction it follows that
\[
\lim\limits_{\xi\to 1,|\xi|<1}(I-\xi Y_{\lambda}^*)^{-1}P_{\sN_{\lambda}}h=(I-Y^*_{\lambda})^{-1}P_{\sN_{\lambda}}h=h\;\; \forall h\in\sL,
\]
where the limit is non-tangential to the unit circle. Hence and from \eqref{snhf} we conclude that \eqref{apr9bb} holds true.

\end{proof}
\begin{remark}\label{may31a}
Since $D_{Y_\lambda}= P_{\sN_\lambda}$, $D_{Y_\lambda^*}=P_{\sN_{\bar\lambda}},$ and $Y_\lambda\uphar\sD_{Y_\lambda}=0$, the function
\[
\Theta_{ Y_{\lambda}}(\zeta)=\zeta\,P_{\sN_{\bar\lambda}}\big (I- \zeta\, Y^*_{\lambda}\big)^{-1}\uphar\sN_{\lambda},\;|\zeta|<1
\]
is the  Sz.-Nagy--Foias characteristic function of $Y_{\lambda}$ (see \cite[Chapter VI, (1.1)]{SF}).

A connection between Sz.-Nagy -- Foias characteristic functions and the characteristic function of a densely defined closed symmetric operator with equal deficiency indices has been established in \cite{Kochubei} in the framework of a boundary value spaces approach.
\end{remark}
\begin{remark}\label{apr10a}
If $S$ is a maximal symmetric operator with $n_+>0,$ $ n_-=0$, then $\wt S_z=S^*$ and $\wt S^*_z=S$ for all $z\in\dC_+$. Hence for each $\lambda\in\dC_+$ one has
$$C^{S}_{\lambda}(z)=0:\sN_{\lambda}\to\{0\}\;\;\forall z\in\dC_+.$$
If $S$ is a maximal symmetric operator with $n_+=0,$ $ n_->0$, then $\wt S_z=S$ and $\wt S^*_z=S^*$ for all $z\in\dC_+$. It follows that for each $\lambda\in\dC_+$
$$C^{S}_{\lambda}(z)=0:\{0\}\to\sN_{\bar{\lambda}}\;\;\forall z\in\dC_+.$$
\end{remark}
\begin{remark}\label{domenik}
The concept of the characteristic function of a linear operator in a Hilbert space was first introduced by M.S.~Liv\v{s}ic in \cite{livsic} (see \cite{AG}).
For a densely defined symmetric operator with deficiency indices $\left<1,1\right>$ the Liv\v{s}ic characteristic function takes the form
\[
w(z)=\cfrac{z-i}{z+i}\,\,\cfrac{(\f_z, e_-)}{(\f_z,e_{+})},\; z\in\dC_+,
\]
where $e_{\pm}\in\sN_{\pm i}$, $||e_+||=||e_-||=1$, $\f_z\in\sN_z$.
Taking into account \eqref{pes11a} we conclude that $w(z)$ coincides with $C^{S}_{i}(z)$
($\lambda=i$). 

Further investigations of the  Liv\v{s}ic characteristic functions and their applications can be founded in \cite{maktsek, maktsek2022}.
\end{remark}
\begin{remark}\label{pes12a}
Let $\lambda=a+ib\in\dC_+$. Set
$
S_{a,b}:=b^{-1}(S-aI).
$ 
Then for deficiency subspaces of $S$ and $S_{a,b}$ one gets the equality
$ 
\sN_{\xi}(S_{a,b})=\sN_{a+b\xi}(S).
$ 
It follows from \eqref{pes11a} that
\[
C_\lambda^S(z)=C_i^{S_{a,b}}(b^{-1}(z-a)).
\]
\end{remark}

\section{Contractions and their defect operators} \label{apr2}
Here we consider properties of contractions  related to mutual location of certain linear manifolds connected with them. We apply results established here in our study of compressions in Sections \ref{apr11b} and \ref{apr11c}.
\begin{theorem}\label{mar14a}
Let $Y$ be a contraction in the Hilbert space $\sH$. Assume $\ker (I-Y)=\{0\}$.
Let $Y_R:=\half(Y+Y^*)$ be the real part of $Y$.
Then
\begin{enumerate}
\item the following equalities
\begin{equation}\label{mar15a}
\ran D_{Y^*}+\ran (I-Y)=\ran D_{Y}+\ran (I-Y^*)=\ran (I-Y_R)^\half
\end{equation}
hold, therefore, the inclusions
\[
\ran(I-Y), \;\ran (I-Y^*),\;\ran D_Y,\;\ran D_{Y^*}\subseteq\ran (I-Y_R)^\half
\]
are valid;
\item the following are equivalent:
\begin{enumerate}
\def\labelenumi{\rm (\roman{enumi})}
\item $\ran (I-Y)\cap\ran D_{Y^*}=\{0\}$,
\item $\ran (I-Y^*)\cap\ran D_{Y}=\{0\};$
\end{enumerate}
\item the equality
\[
\ran(I-Y^*)\cap\ran D_Y=(I-Y^*)(I-Y)^{-1}\left(\ran(I-Y)\cap\ran D_{Y^*}\right)
\]
 holds;
\item if
 $D_{Y^*}h\in \ran (I-Y)$, then
\[
\left\|-Y^*h+D_{Y}(I-Y)^{-1}D_{Y^*}h\right\|=||h||
\]
and
\[
D_{Y}\left(-Y^*h+D_{Y}(I-Y)^{-1}D_{Y^*}h\right)\in\ran (I-Y^*);
\]
\item the following are equivalent:
\begin{enumerate}
\def\labelenumi{\rm (\roman{enumi})}
\item 
$\ran D_{Y^*}\subseteq\ran (I-Y),$ 
\item 
$\ran(I-Y)=\ran (I-Y_R)^\half$; 
\end{enumerate}
\item the following are equivalent:
\begin{enumerate}
\def\labelenumi{\rm (\roman{enumi})}
\item
$\ran D_{Y}\subseteq\ran (I-Y^*),$ 
\item 
$\ran(I-Y^*)=\ran (I-Y_R)^\half;$ 
\end{enumerate}
\item the following are equivalent:
\begin{enumerate}
\def\labelenumi{\rm (\roman{enumi})}
\item
$\left\{\begin{array}{l}
\ran D_{Y^*}\subseteq\ran (I-Y)\\
\ran D_{Y}\subseteq\ran (I-Y^*)
\end{array}\right.,$
\item $\ran (I-Y^*)=\ran(I-Y)= \ran (I-Y_R)^\half;$
\end{enumerate}
\item the following are equivalent:
\begin{enumerate}
\item $\ran D^2_{Y}\subset\ran (I-Y^*)$,
\item $\ran (I-Y)\subseteq\ran (I-Y^*).$
\end{enumerate}
\end{enumerate}

\end{theorem}

\begin{proof}
(1) (see \cite[Proposition 3.1]{Arl_arxiv2024}). We use the following equality for bounded operators $F,G\in\bB(\cH)$ (see \cite[Theorem 2.2]{FW}):
\[
\ran F+\ran G=\ran \left(FF^*+GG^*\right)^\half.
\]
Hence
\[
\begin{array}{l}
\ran (I-Y^*)+\ran D_{Y}=\ran\left((I-Y^*)(I-Y)+D^2_{Y}\right)^\half\\
=\ran\left((I-Y^*)(I-Y)+I-Y^* Y\right)^\half=\ran\left(2I- (Y^*+ Y))\right)^\half.
\end{array}
\]
Similarly
\[
\ran (I-Y)+\ran D_{Y^*}=\ran\left(2I- (Y^*+ Y))\right)^\half.
\]

(2) (see \cite[Proposition 3.5]{ArlCAOT2023}). We give a proof for completeness. Suppose that $ D_{Y^*} h=(I-Y)g.$
Then due to the equality $Y^* D_{Y^*}=D_{Y} Y^*$
we get
\begin{multline*}
Y^*D_{Y^*} h= D_{Y}  Y^*h=Y^*(I-Y)g\\
= D^2_{Y} g-(I-Y^*)g\Longleftrightarrow   D_{Y}\left( D_{Y}g-Y^* h\right)=(I-Y^*)g.
\end{multline*}
It follows that
$$\ran  D_{Y^* }\cap\ran (I-Y)\ne\{0\}\Longrightarrow \ran D_{Y}\cap\ran (I-Y^*)\ne\{0\}.$$
Similarly
$$\ran D_{Y}\cap\ran (I-Y^*)\ne\{0\}\Longrightarrow \ran  D_{Y^*}\cap\ran (I-Y)\ne\{0\}.$$
(3) Suppose $h=(Y-I)z=D_{Y^*}x$. Then $(Y-I)^{-1}h=z$, $ (Y^*-I)(Y-I)^{-1}h=(Y^*-I)z$.
\[
\begin{array}{l}
(Y-I)z=D_{Y^*}x\Longrightarrow Y^*(Y-I)z=Y^*D_{Y^*}x=D_Y Y^*x\Longrightarrow (Y^*Y-Y^*)z=D_Y Y^*x\\
\Longrightarrow -D^2_Yz-(Y^*-I)z=D_Y Y^*x\Longrightarrow (Y^*-I)z=-D_Y(D_Yz+Y^*x)\\
\Longrightarrow (Y^*-I)(Y-I)^{-1}h\in\ran D_Y\cap\ran (Y^*-I).
\end{array}
\]
Thus, $(I-Y^*)(I-Y)^{-1}\left(\ran (I-Y)\cap\ran D_{Y^*}\right)\subseteq\left(\ran (I-Y^*)\cap\ran D_{Y}\right)$.
The converse inclusion follows from the replacing $Y\mapsto Y^*.$

(4) The equality $D_{Y^*}h=0$ is equivalent to the equality $||Y^*h||=||h||$ and in this case $-Y^*h+D_{Y}(I-Y)^{-1}D_{Y^*}h=-Y^*h$. Suppose $D_{Y^*}h\ne 0$ and $D_{Y^*}h\in \ran (I-Y)$. Using the equalities $D_Y Y^*=Y^*D_{Y^*}$ \cite{SF} and $||D_Yg||^2=||g||^2-||Yg||^2$, we get
\[\begin{array}{l}
\left\|-Y^*h+D_{Y}(I-Y)^{-1}D_{Y^*}h\right\|^2=||Y^*h||^2+\left\|D_{Y}(I-Y)^{-1}D_{Y^*}h\right\|^2\\[2mm]
-2\RE(Y^*h, D_{Y}(I-Y)^{-1}D_{Y^*}h)=||Y^*h||^2+\left\|(I-Y)^{-1}D_{Y^*}h\right\|^2\\[2mm]
-\left\|{Y}(I-Y)^{-1}D_{Y^*}h\right\|^2-2\RE(D_{Y^*}h, Y(I-Y)^{-1}D_{Y^*}h)\\[2mm]
=||h||^2-||D_{Y^*}h||^2 -\left\|{Y}(I-Y)^{-1}D_{Y^*}h\right\|^2-2\RE(D_{Y^*}h, Y(I-Y)^{-1}D_{Y^*}h)\\[2mm]
+\left\|(I-Y)^{-1}D_{Y^*}h\right\|^2
=||h||^2-||D_{Y^*}h+{Y}(I-Y)^{-1}D_{Y^*}h||^2+\left\|(I-Y)^{-1}D_{Y^*}h\right\|^2\\[2mm]
=||h||^2-||D_{Y^*}h+(Y-I)(I-Y)^{-1}D_{Y^*}h+(I-Y)^{-1}D_{Y^*}h||^2\\[2mm]
+\left\|(I-Y)^{-1}D_{Y^*}h\right\|^2
=||h||^2-\left\|(I-Y)^{-1}D_{Y^*}h\right\|^2+\left\|(I-Y)^{-1}D_{Y^*}h\right\|^2=||h||^2.
\end{array}
\]
Further
\[
\begin{array}{l}
D_{Y}\left(-Y^*h+D_{Y}(I-Y)^{-1}D_{Y^*}h\right)=-Y^*D_{Y^*}h+D^2_{Y}(I-Y)^{-1}D_{Y^*}h\\[2mm]
=\left(-Y^*+(I-Y^*Y)(I-Y)^{-1}\right)D_{Y^*} h=\left(-Y^*+(Y^*(I-Y)-Y^*+I)(I-Y)^{-1}\right)D_{Y^*} h\\[2mm]
=(I-Y^*)(I-Y)^{-1}D_{Y^*}h\in\ran (I-Y^*).
\end{array}
\]

(5), (6), (7). Since equalities in \eqref{mar15a} yield inclusions
$$\ran (I-Y),\ran (I-Y^*)\subseteq\ran (I-Y_R)^\half,$$
the inclusion $\ran D_{Y^*}\subseteq\ran (I-Y)$ (respectively, $\ran D_{Y}\subseteq\ran (I-Y^*)$) together with \eqref{mar15a} imply that $\ran (I-Y)=\ran (I-Y_R)^\half$ (respectively,  $\ran (I-Y^*)=\ran (I-Y_R)^\half$).

Conversely, assume that $\ran (I-Y)=\ran (I-Y_R)^\half$. Then the Douglas lemma \cite{Doug} yields that
$I-Y=(I-Y_R)^\half X$ with bounded operator $X$ and therefore
\[
(I-Y)(I-Y^*)=(I-Y_R)^\half XX^*(I-Y_R)^\half \Longleftrightarrow D^2_{Y^*}=(I-Y_R)^\half (2- XX^*)(I-Y_R)^\half.
\]
It follows that
$$\ran D_{Y^*}=(I-Y_R)^\half\ran(2-XX^*)^\half\subset\ran(I-Y_R)^\half=\ran (I-Y).$$
Similarly the inclusion $\ran (I-Y^*)=\ran (I-Y_R)^\half$ implies $\ran D_{Y}\subset \ran (I-Y_R)^\half.$

(8) If for each $f\in\sH$ one can find $h\in \sH$ such that $(I-Y)f=(I-Y^*)h$, then the equality
\[
I-Y+(I-Y^*)Y=D^2_{Y}
\]
implies
$
D^2_{Y}f=(I-Y^*)Yf+(I-Y^*)h\in\ran (I-Y^*).
$
Conversely, if  $D^2_{Y}f\in\ran (I-Y^*)$ for each $f\in\sH$,   then
\[
(I-Y)f=D^2_{Y}f-(I-Y^*)Y f\in\ran (I-Y^*) .
\]
Thus
\[
\ran D^2_{Y^*}\subseteq\ran (I-Y)\\
\Longleftrightarrow \ran (I-Y)=\ran (I-Y^*).
\]
\end{proof}
\begin{remark}\label{apr26aa}
Since $\ran F=\ran(FF^*)^\half$, the following are valid:
\[
\begin{array}{l}
(1)\; D_Y=0\Longleftrightarrow \; Y \;\mbox{is an isometry}\\[2mm]
\Longrightarrow\ran (I-Y^*)=\ran ((I-Y^*)(I-Y))^\half=\ran (I-Y_R)^\half, \\[2mm]
(2)\; D_{Y^*}=0 \Longleftrightarrow\; Y^* \;\mbox{is an isometry\;(i.e.,\; Y\; is a co-isometry)}\\
\Longrightarrow\ran (I-Y^*)=\ran(I-Y_R)^\half,\\[2mm]
(3)\; D_{Y}=0\;\mbox{and}\;D_{Y^*}=0\\[2mm]
\Longleftrightarrow \; Y \;\mbox{is a unitary}\Longrightarrow \ran(I-Y)=\ran (I-Y^*)=\ran(I-Y_R)^\half.
\end{array}
\]
\end{remark}
\begin{remark}\label{juni22a}
Due to  Theorem \ref{ivekmzy}, Proposition \ref{apr7a} and equalities in \eqref{30jul} and \eqref{mar16bb}
 all cases related to contractions in Theorem \ref{mar14a} can be realized by means of the Cayley transforms of the Shtraus extensions of symmetric operators in an infinite-dimensional separable complex Hilbert space.

 Besides it is shown in \cite[Corollary 6.2]{Arl_arxiv2024} that there exist contractions $Y$ such that
\[
\begin{array}{l}
\ker D_Y=\{0\},\;\; \ran (I-Y)\cap \ran D_{Y^*}=\{0\}(\Longleftrightarrow \ran (I-Y^*)\cap\ran D_Y=\{0\}),\\[2mm]
\ran (I-Y)\cap\ran (I-Y^*)=\{0\}.
\end{array}
\]
\end{remark}
\section{Compressions of selfadjoint extensions}\label{mar16}
\begin{proposition}
\label{bef15b}
Let $S$,
 be a non-densely defined regular symmetric operator $S$ with equal deficiency indices, let $U\in\bB(\sN_i,\sN_{-i})$ be an admissible unitary operator,
 and let $\wt S$ be the corresponding selfadjoint extensions of $S$, i.e.,
\begin{equation}\label{30bjul}
\begin{array}{l}
\dom \wt S=\dom S\dot+(I-U)\sN_i,\; \ker(U-V_i)\uphar\sL_i=\{0\},\\[2mm]
 \wt S(f_S+\f_i-U\f_i)=Sf_S+i\f_i+iU\f_i,\; f_S\in\dom S,\;\f_i\in\sN_i.
 \end{array}
\end{equation}
Then
\begin{enumerate}
\item the equalities
\begin{equation}\label{bef8}
\left\{\begin{array}{l}
\dom \wt S\cap \cH_0=\dom S\dot+(I-U)\ker(P_{\sL_{-i}}U- V_iP_{\sL_i})\\[2mm]
=\dom S\dot + (I-U^{-1})\ker (P_{\sL_{i}}U^{-1}-V_{-i}P_{\sL_{-i}})\\[2mm]
\qquad\qquad =\dom S\dot+ (I-P_{\sN'_{-i}}U)P_{\sN'_i}\ker(P_{\sL_{-i}}U- V_iP_{\sL_i})\\[2mm]
= \dom S\dot+ (I-P_{\sN'_{i}}U^{-1})P_{\sN'_{-i}}\ker(P_{\sL_{i}}U^{-1}-V_{-i} P_{\sL_{-i}})
\end{array}\right.
\end{equation}
 hold,
where $\sL_{\pm i}$ and $V_{\pm i}$ are defined in \eqref{gjlghjcn} and \eqref{forbis}, respectively;
\item the subspaces $\sN'_i\ominus P_{\sN'_i}\ker(P_{\sL_{-i}}U- V_iP_{\sL_i})$ and\\
$\sN'_{-i}\ominus P_{\sN'_{-i}}\ker(P_{\sL_{i}}U^{-1}-V_{-i} P_{\sL_{-i}})$
are the deficiency subspaces of $\wt S_0:=P_{\cH_0}\wt S\uphar{\dom \wt S\cap \cH_0}$;
\item the following are equivalent:
\begin{enumerate}
\def\labelenumi{\rm (\roman{enumi})}
\item the operator $\wt S_0$ is closed,
\item the linear manifold $P_{\sN'_i}\ker(P_{\sL_{-i}}U- V_iP_{\sL_i})$ is a subspace,
\item the linear manifold $P_{\sN'_{-i}}\ker(P_{\sL_{i}}U^{-1}-V_{-i}P_{\sL_{-i}} )$ is a subspace;
\end{enumerate}
\item the following are equivalent:
\begin{enumerate}
\def\labelenumi{\rm (\roman{enumi})}
\item $\dom\wt S\cap\cH_0=\dom S,$
\item $\ker(P_{\sL_{-i}}U- V_iP_{\sL_i})=\{0\},$
\item $P_{\sN'_i}\ker(P_{\sL_{-i}}U- V_iP_{\sL_i})=\{0\}$,
\item $\ker (P_{\sL_{i}}U^{-1}-V_{-i}P_{\sL_{-i}})=\{0\}$,
\item $P_{\sN'_{-i}}\ker(P_{\sL_{i}}U^{-1}-V_{-i}P_{\sL_{-i}})=\{0\};$
\end{enumerate}
in particular,
if one of the semi-deficiency subspaces of $S$ is trivial, then $\dom \wt S\cap \cH_0=\dom S$;
\item
in the case $n_+'=n_-'\ne 0$ the following are equivalent:
\begin{enumerate}
\def\labelenumi{\rm (\roman{enumi})}
\item operator $\wt S_0$ is selfadjoint in $\cH_0,$
\item $P_{\sN'_i}\ker(P_{\sL_{-i}}U- V_iP_{\sL_i})=\sN'_i$ and $P_{\sN'_{-i}}U\ker(P_{\sL_{-i}}U- V_iP_{\sL_i})=\sN'_{-i},$
\item $P_{\sN'_{-i}}\ker(P_{\sL_{i}}U^{-1}-V_{-i}P_{\sL_{-i}})=\sN'_{-i}$ and $P_{\sN'_{i}}U^{-1}P_{\sN'_{-i}}\ker(P_{\sL_{i}}U^{-1}-V_{-i}P_{\sL_{-i}})=\sN'_i$,
\item $P_{\sN'_i}\ker(P_{\sL_{-i}}U- V_iP_{\sL_i})=\sN'_i$ and $P_{\sN'_{-i}}\ker(P_{\sL_{i}}U^{-1}-V_{-i}P_{\sL_{-i}})=\sN'_{-i}$.
\end{enumerate}
\end{enumerate}
\end{proposition}
\begin{proof}
Note that if $\dim\sL< \infty$, then $S$ is regular and the semi-deficiency indices are equal.

Set
\begin{equation}
\label{ee3-34} \ti\sN_i:=\{\varphi_i\in\sN_i,
(U-I)\varphi_i\in\cH_0\},\;\ti\sN_{-i}:=\{\phi_{-i}\in\sN_{-i}:(U^{-1}-I)\phi_{-i}\in\cH_0\}.
\end{equation}
Clearly
\begin{equation}\label{mar5c}
\begin{array}{l}
\ti\sN_{-i}=U\ti\sN_{i},\\
\dom\wt S\cap\cH_0=\dom S\dot+(I-U)\ti\sN_i=\dom S \dot +(I-U^{-1})\ti\sN_{-i}.
\end{array}
\end{equation}
Suppose that $(I-U)\f_i\in\dom \wt S\cap \cH_0$, i.e, $\f_i\in\ti \sN_i.$
Then
\[
\begin{array}{l}
(I-U)\f_i\in\dom S^*\cap\cH_0=\dom S\dot +\sN'_i\dot +\sN'_{-i}\\
\Longleftrightarrow (I-U)\f_i=f_S+\f_i'+\f_{-i}',\; f_S\in\dom S,\f_i'\in\sN'_i,\; \f_{i}'\in\sN_{-i}'.
\end{array}
\]
From \eqref{gthtctx1} we get that
\[
\f_i'-\f_i=P_{\sN_i}h,\; \f_{-i}'+U\f_i=-P_{\sN_{-i}}h,\; h\in\sL.
\]
Hence
\[
P_{\sL_i}\f_i=-P_{\sN_{i}}h,\; P_{\sL_{-i}}U\f_i=-P_{\sN_{-i}}h=-V_iP_{\sN_i}h,\; h\in\sL.
\]
Thus $P_{\sL_{-i}}U\f_i=-V_iP_{\sN_{-i}}h=V_i P_{\sL_i}\f_i$, i.e., $\f_i\in \ker(P_{\sL_{-i}}U- V_iP_{\sL_i}).$

Conversely, if $\f_i\in \ker(P_{\sL_{-i}}U- V_{i}P_{\sL_i})$, then $P_{\sL_{-i}}U\f_i=V_iP_{\sL_i}\f_i$ and from \eqref{gjlghjcn}, \eqref{forbis} there is $h\in\sL$ such that $P_{\sL_i}\f_i=P_{\sN_{i}}h$ and
\[
\f_i=P_{\sN'_i}\f_i+P_{\sN_{i}}h, \; U\f_i= P_{\sN'_{-i}}U\f_i+P_{\sN_{-i}}h.
\]
Hence and from \eqref{xthdjy}
\[
\begin{array}{l}
\f_i-U\f_i=P_{\sN'_i}\f_i-P_{\sN'_{-i}}U\f_i+P_{\sN_{i}}h-P_{\sN_{-i}}h=P_{\sN'_i}\f_i-P_{\sN'_{-i}}U\f_i\\
+(I_\cH-P_{\sM_{i}})h-(I_\cH-P_{\sM_{-i}})h\\
=P_{\sN'_i}\f_i-P_{\sN'_{-i}}U\f_i -P_{\sM_{i}}h+(S+i I)(S-i I)^{-1}P_{\sM_i}h\\
=P_{\sN'_i}\f_i-P_{\sN'_{-i}}U\f_i+2i(S-i I)^{-1}P_{\sM_i}h\in \cH_0.
\end{array}
\]
So, it is proved that
\begin{equation}\label{mar5}
\ti\sN_i=\ker(P_{\sL_{-i}}U- V_iP_{\sL_i}).
\end{equation}
Arguing similarly we get (recall that $V_{-i}=V^{-1}_i=V^*_i$, see \eqref{forbis}) that
\begin{equation}\label{march17}
\ti\sN_{-i}=\ker (P_{\sL_{i}}U^{-1}-V_{-i}P_{\sL_{-i}}),
\end{equation}
\begin{multline*}
\dom\wt S\cap \cH_0=\dom S\dot + (I-U^{-1})\ker (P_{\sL_{i}}U^{-1}-V_{-i}P_{\sL_{-i}})\\
=\dom S\dot+ (I-P_{\sN'_{i}}U^{-1})P_{\sN'_{-i}}\ker(P_{\sL_{i}}U^{-1}-V_{-i}P_{\sL_{-i}}),
\end{multline*}
Thus, \eqref{bef8} is proved.
Observe that the mapping
\begin{equation}\label{bef8b}\left\{\begin{array}{l}
\dom U_0=P_{\sN'_i}\ker(P_{\sL_{-i}}U- V_iP_{\sL_i}),\\[2mm]
U_0:P_{\sN'_i}\f_i\mapsto P_{\sN'_{-i}}U\f_i,\; \f_i\in \ker(P_{\sL_{-i}}U- V_iP_{\sL_i})
\end{array}\right.
\end{equation}
is an isometry. Actually, if $P_{\sL_{-i}}U\f_i= V_iP_{\sL_i}\f_i$, $\f_i\in\sN_i$, then
\[
\begin{array}{l}
||P_{\sN'_{-i}}U\f_i||^2=||U\f_i||^2-||P_{\sL_{-i}}U\f_i||^2=||\f_i||^2-||V_iP_{\sL_i}\f_i||^2\\[2mm]
=||\f_i||^2-||P_{\sL_i}\f_i||^2=||P_{\sN'_i}\f_i||^2.
\end{array}
\]
Besides, it has been shown above that
\begin{equation}\label{bef8c}
\dom \wt S\cap \cH_0=\dom S\dot+ (I-U_0)\dom U_0.
\end{equation}
Therefore
$$
\cH_0\ominus(\wt S_0+iI)(\dom\wt S_0)=\sN'_i\ominus P_{\sN'_i}\ker(P_{\sL_{-i}}U- V_iP_{\sL_i}).
$$
It is evident that
\[
\f_i\in \ker(P_{\sL_{-i}}U- V_iP_{\sL_i})\Longleftrightarrow P_{\sL_{-i}}UP_{\sN_i}\f_i=(V_i- P_{\sL_{-i}}U) P_{\sL_{i}}\f_i.
\]
Therefore
\begin{equation}\label{bef19a}
\ker(P_{\sL_{-i}}U- V_iP_{\sL_i})=\left \{\f'_i+f_i:\f'_i\in\sN'_i,\; f_i\in\sL_i,\;P_{\sL_{-i}}U\f'_i=(V_i- P_{\sL_{-i}}U) f_i\right\}.
\end{equation}
It follows that along with \eqref{bef8b} the following is valid:
\[
\dom U_0 =\left\{\f'_i\in\sN'_i:P_{\sL_{-i}}U\f'_i\in (V_i- P_{\sL_{-i}}U)\sL_i\right\}
\]
and hence
\[
\dom U_0=\sN'_i\Longleftrightarrow P_{\sL_{-i}}U\sN'_i\subset (V_i- P_{\sL_{-i}}U)\sL_i.
\]

The inverse $U_0^{-1}$ can be described as follows
\begin{equation}\label{bef8bb}
\left\{\begin{array}{l}
\dom U_0^{-1}=P_{\sN'_{-i}}\ker(P_{\sL_{i}}U^{-1}- V_{-i}P_{\sL_{-i}}),\\
U_0^{-1}:P_{\sN'_{-i}}\f_{-i}\mapsto P_{\sN'_{i}}U^{-1}\f_{-i},\; \f_{-i}\in \ker(P_{\sL_{i}}U^{-1}- V_{-i}P_{\sL_{-i}}).
\end{array}\right.
\end{equation}
Consequently
\[
\cH_0\ominus(\wt S_0-iI)(\dom\wt S_0)=\sN'_{-i}\ominus P_{\sN'_{-i}}\ker(P_{\sL_{i}}U^{-1}-V_{-i} P_{\sL_{-i}}).
\]
Equivalences in (3), (4), (5) follow from \eqref{bef8}, \eqref{bef8b}, \eqref{bef8c}, \eqref{bef8bb}.

Suppose that $\sN'_i=\{0\}$. Then due to $\ker (U-V_i)=\{0\},$ we get from \eqref{bef8} that $\dom \wt S \cap \cH_0=\dom S.$
Similarly the equality $\sN'_{-i}=\{0\}$ implies $\dom \wt S \cap \cH_0=\dom S.$

\end{proof}

\begin{remark}\label{mar23a}
If $n_+=n_-$ and $n_\pm'\ne 0$, then a unitary operator $U\in\bB(\sN_i,\sN_{-i})$ can be written in the block-operator matrix form
\[
U=
\begin{bmatrix} U_{11} &U_{12} \cr
U_{21}& U_{22}\end{bmatrix}:\begin{array}{l}\sL_{i}\\ \oplus\\
\sN'_i\end{array}\to \begin{array}{l}\sL_{-i}\\ \oplus\\
\sN'_{-i}\end{array}\Longleftrightarrow U^{-1}=U^*=
\begin{bmatrix} U_{11}^* &U_{21}^* \cr
U_{12}^*& U_{22}^*\end{bmatrix}:\begin{array}{l}\sL_{-i}\\ \oplus\\
\sN'_{-i}\end{array}\to \begin{array}{l}\sL_{i}\\ \oplus\\
\sN'_{i}\end{array}.
\]
It follows from \eqref{bef8b} and \eqref{bef8bb} that the operator $U_0$ and $U^{-1}_0$ are the lower right corners of the Schur complements of the operator matrices $U-V_iP_{\sL_i}$ and $U^{*}-V_{-i}P_{\sL_{-i}}$, i.e.,
\[
\begin{array}{l}
U_0=U_{22}-U_{21}(U_{11}-V_i)^{-1}U_{12},\; \dom U_0=\left\{\f_i'\in\sN_i':U_{12}\f_{i}'\in\ran (U_{11}-V_i)\right\},\\[2mm]
U^{-1}_0=U_{22}^*-U_{12}^*(U_{11}^*-V_{-i})^{-1}U_{21}^*,\; \dom U^{-1}_0=\left\{\f_{-i}'\in\sN_{-i}':U_{21}^*\f_{-i}'\in\ran (U_{11}^*-V_{-i})\right\}.
\end{array}
\]

\end{remark}
\begin{corollary}\label{bef7}
Let $S$ be a non-densely defined closed symmetric operator with equal deficiency indices. Let $\wt S$ be a selfadjoint extension of $S$ and let $U\in\bB(\sN_i,\sN_{-i})$ be an admisssible unitary operator determining $\wt S$ by \eqref{feb7} ($\lambda=i$).
Assume  $\dim\sL<\infty$.
Then
the compression $\wt S_0=P_{\cH_0}\wt S\uphar(\dom \wt S\cap \cH_0)$ is selfadjoint in $\cH_0$.
\end{corollary}
\begin{proof}
Because $\sL$ is finite-dimensional, by Theorem \ref{ivekmzy}
 the operator $S_0$ is closed.
 Since $V_if\ne Uf$ for all $f\in\sL_i\setminus\{0\}$, the operators $V_i$ and $U$ are isometries and due to \eqref{KRASn}
$$\dim\sL_i=\dim P_{\sN_i}\sL=\dim \sL<\infty,$$
we get that the operator $(V_i- P_{\sL_{-i}}U)\uphar\sL_i$ has bounded inverse defined on $\sL_{-i}$ and
\[
P_{\sL_{i}}\f_i=(V_i- P_{\sL_{-i}}U)^{-1}P_{\sL_{-i}}UP_{\sN_i}\f_i.
\]
Hence
\[
\begin{array}{l}
\ker(P_{\sL_{-i}}U- V_iP_{\sL_i})=\left\{\f'_i + (V_i- P_{\sL_{-i}}U)^{-1}P_{\sL_{-i}}U\f'_i:\;\f'_i\in \sN'_i\right\}\\
\qquad\qquad=\left(I+(V_i- P_{\sL_{-i}}U)^{-1}P_{\sL_{-i}}U\right)\sN'_i.
\end{array}
\]
It follows that the domain of the operator $U_0$ defined in \eqref{bef8b} coincides with $\sN'_i$.

Hence
\[
P_{\sN'_{-i}}\ker(V_{-i} P_{\sL_{-i}}-P_{\sL_{i}}U^{-1})=\sN'_{-i}
\]
and
\begin{equation}\label{bef9a}
\dom \wt S\cap \cH_0=\dom S\dot+ (I-P_{\sN'_{i}}U^{-1})P_{\sN'_{-i}}\ker(V_{-i} P_{\sL_{-i}}-P_{\sL_{i}}U^{-1}).
\end{equation}
 The equalities \eqref{bef8c} and \eqref{bef9a} show that the operator $U_0$ defined in \eqref{bef8b} unitarily maps $\sN'_i$ onto $\sN'_{-i}$, its inverse $U^{-1}_0$ is the operator defined in \eqref{bef8bb}.
Hence, by Proposition \ref{bef15b} (5), the compression $\wt S_0=
P_{\cH_0}\wt S\uphar(\dom \wt S\cap H_0)$ is a selfadjoint extension of the operator $S_0$ in $\cH_0$.

\end{proof}
\begin{corollary}(Stenger's lemma \cite[Lemma 1]{Stenger}).
Let $\wt S$ be a selfadjoint operator in $\cH$ and let $\cH_0$ be a finite-co-dimensional subspace of $\cH$. Then the operator $P_{\cH_0}\wt S\uphar(\dom \wt S\cap \cH_0)$ is selfadjoint in $\cH_0$.
\end{corollary}
\begin{proof}
Define the linear operator $S$ in $\cH$:
\[
S:=\wt S\uphar(\dom \wt S\cap\cH_0). 
\]
Then $\dom S$ is dense in $\cH_0$ $S$ is a closed regular  symmetric operator in $\cH$ and $\wt S$ is a selfadjoint extension of $S$.
By Theorem \ref{bef7} the operator $\wt S_0=P_{\cH_0}\wt S\uphar(\dom \wt S\cap \cH_0)$ is selfadjoint in $\cH_0$.
\end{proof}
\begin{theorem}\label{t3-15} cf. \cite[Theorem 2.5.7]{ArlBelTsek2011}
Let $\wt S$ be a regular self-adjoint
 extension of a non-densely defined regular symmetric operator $S$. Then the operator $P_{\cH_0}\wt S\uphar{(\dom \wt S\cap\cH_0)}$ is a self-adjoint extension of the operator $S_0$ in
$\cH_0$ and, therefore, the semi-deficiency indices of $S$ are equal.
 \end{theorem}
\begin{proof}
Let $U$ be  an admissible unitary operator $U\in\bB(\sN_i,\sN_{-i}) $ that
determines $\wt S$ by von Neumann' formulas \eqref{30bjul}  and let $\ti\sN_i$ be defined in \eqref{ee3-34}.
Let us define the following linear manifold
\[
\sN^0_i:=(\wt S-iI)^{-1}\sL.
\]
Then $\sN^0_i\subset\sN_i$. Indeed, if  $f\in\sL$ and $g\in\sM_{-i}=
(S+iI)\dom S$, then
$$
(g,(\wt S-iI)^{-1}f)=((\wt S+iI)^{-1}g,f)=((S+i I)^{-1}g,f)=0,
$$
Since $\wt S$ is a regular self-adjoint extension of $S$, the linear manifold $\dom\wt S$ is a subspace in $\cH^+_{S^*}$ and the operator
 and therefore $(\wt S-iI)^{-1}$ has the estimate
\[
c_1||f||\le||(\wt S-i I)^{-1}f||_{\cH^+_{S^*}}\le c_2||f||\quad\mbox{for all}\quad
f\in\cH,
\]
with $c_1>0$. Hence $\sN^0_i$ is a subspace in $\cH^+_{S^*}$ (and
simultaneously in $\cH).$

Let $\ti\sN_i$ be defined in \eqref{ee3-34}. Let us show that
\begin{equation}
\label{ee3-35} \sN_i^0\oplus\ti\sN_i=\sN_i,
\end{equation}
where the sum is orthogonal in $\cH.$ We have from \eqref{30bjul} 
$$
(\wt S+iI)^{-1}\varphi=-\frac{1}{2i}(U-I)\varphi,\; \varphi\in\sN_i.
$$
This yields the equivalences
\[
(U-I)\varphi\in\cH_0\iff(\wt S+i I)^{-1}\varphi\perp\sL\iff
\varphi\perp\sN^0_i,
\]
and the decomposition \eqref{ee3-35} holds true.

Suppose the vector $\psi\in\cH_0$ is orthogonal in $\cH_0$ to
$P_{\cH_0}(\wt S+iI)(\dom\wt S\cap\cH_0)$. Then $\psi\in\sN_i\cap\cH_0=\sN_{i}'$. Because
\[
(\wt S+iI)(\dom \wt S\cap \cH_0)\supset(\wt S+i I)(I-U)\ti\sN_i=\ti\sN_i
\]
and $\psi\in\cH_0$, we get from \eqref{ee3-35} that $\psi\in \sN_i^0.$
Thus $\psi\in\sN'_i\cap\sN^0_i.$
Since
$\sN^0_i\subset\dom\wt S$, we get that $\psi\in\dom\wt S$. Hence (see
\eqref{bef15a}) $P_{\cH_0}\wt S\psi=S^*\psi=i\psi$. Because
$P_{\cH_0}\wt S\uphar(\dom\wt S)\cap\cH_0$ is symmetric, we conclude that $\psi=0$. This
means $(P_{\cH_0}\wt S+iI)(\dom\wt S\cap\cH_0)$ is dense in $\cH_0$. Similarly it
can be proved that $(P_{\cH_0}\wt S-iI)(\dom \wt S\cap\cH_0)$ is dense in $\cH_0$.
Thus, the operator $P_{\cH_0}\wt S\uphar(\dom\wt S\cap\cH_0)$ is essentially
self-adjoint in $\cH_0$. On the other hand, since $\wt S$ is a regular
extension, the linear manifold
$\dom\wt S\cap\cH_0$ is a subspace in $\cH^+_{S^*}$ (see \eqref{aug12a}) and hence $P_{\cH_0}\wt S\uphar(\dom S\cap\cH_0)$ is
self-adjoint in $\cH_0$.
\end{proof}
So, if $\wt S$ is a regular selfadjoint extension of a regular non-densely defined symmetric operator $S$ and a unitary $U\in\bB(\sN_i,\sN_{-i})$ determines $\wt S$, then Theorem \ref{t3-15} and \eqref{ee3-34}, \eqref{bef8b} and \eqref{bef19a} show that
\[
\begin{array}{l}
\wt\sN_{i}=\{\f_i\in\sN_i:(U-I)\f_i\in\cH_0\}=\ker(P_{\sL_{-i}}U-V_iP_{\sL_i})=\left(I+(V_i-P_{\sL_{-i}}U)^{-1}P_{\sL_{-i}}U\right)\sN_i',\\[2mm]
P_{\sN'_{i}}\ker(P_{\sL_{-i}}U-V_iP_{\sL_i})=\sN'_{i},\;   P_{\sN'_{-i}}U\ker(P_{\sL_{-i}}U-V_iP_{\sL_i})=\sN'_{-i}.
\end{array}
\]

\section{Parametrization of contractive and unitary $2\times 2$ block operator matrices and compressions of selfadjoint extensions}
\label{apr11b}
\begin{lemma}\label{L1} cf. \cite[Lemma 3.1]{AHS_CAOT}
Let $\sH$, $\sK$, $\sM$, and $\sN$ be Hilbert spaces, and let
$F$ be a linear operator, $ \dom F\subset \sH$, $\ran F\subset\sK$. 
Let the operators
$M\in\bB(\sM,\sH)$ and $K\in\bB(\sK,\sN)$ be contractions,
and let the operator $X\in\bB(\sD_{M},\sD_{K^*})$ be a contraction.

Suppose $M\f\in\dom F$ and $||FM\f||\le ||M\f||.$ Then for
\begin{equation}\label{gee}
W\f:=KFM\f+D_{K^*}XD_{M}\f
\end{equation}
the equality
\begin{equation}
\label{CONTR}
\|\f\|^2- \|W\f\|^2 =||M\f||^2-||{F}M\f||^2+\|D_{X}D_{M}\f\|^2
+ \left\|\left(D_{K}FM-K^*XD_{M} \right)\f\right\|^2
\end{equation}
 holds.
Moreover, if $||FM\f||=||M\f||$ and if $K$, $M^*$, and $X$ are isometries, then
\[
||W\f||=||\f||.
\]
\end{lemma}
\begin{proof}
From \eqref{gee} one gets
\begin{equation}\label{un}
\begin{array}{l}
\|\f\|^2-\|W\f\|^2=\|\f\|^2-\left\|\left(KFM+
D_{K^*}XD_{M}\right)\f\right\|^2  \\
 =\|\f\|^2-\|KFM \f\|^2-\|D_{K^*}XD_{M}\f\|^2
- 2\RE\left(KFM\f, D_{K^*}XD_{M}\f\right).
\end{array}
\end{equation}
The equality $K^*D_{K^*}=D_KK^*$ gives
\[
 \left(KFM\f, D_{K^*}XD_{M}\f\right)=\left(D_{K}FM\f,
 K^*XD_{M}\f\right).
\]
The definition of $D_{K^*}$ and $D_K$ shows that
\[
-\|KFM \f\|^2=\|D_{K}FM \f\|^2-\|FM\f\|^2,\;-\|D_{K^*}XD_{M}f\|^2=-\|XD_{M}f\|^2  +\|{K^*}XD_{M}\f\|^2
\]
Now the righthand side of \eqref{un} becomes
\[
\begin{split}
 &\|\f\|^2-\|FM\f\|^2-\|XD_{M}\f\|^2 \\
 & \hspace{1.5cm}+\|D_{K}FM \f\|^2  +\|{K^*}XD_{M}\f\|^2
-2\RE\left(D_{K}FM\f, K^*XD_{M}\f\right)\\
&=\|\f\|^2
 -\|FM\f\|^2  -\|XD_M \f\|^2
 +\left\|\left(D_{K}FM-K^*XD_{M} \right)\f\right\|^2.
\end{split}
\]
Finally, observe that
\[
\|D_FM\f\|^2=\|M\f\|^2-\|FM\f\|^2,
\quad \|D_XD_M\f\|^2=\|\f\|^2-\|M\f\|^2-\|XD_M\f\|^2.
\]
Hence the proof of \eqref{CONTR} is complete.

If $||FM\f||=||M\f||$ and if $K$, $M^*$, and $X$ are isometric operators, then, using that $D_K=0$, $D_X=0$, $\sD_{K^*}=\ker K^*$ and $\ran X\subseteq\sD_{K^*}$
we get that $||W\f||=||\f||.$

\end{proof}

We will use further the following parametrization of $2\times 2$ block-operator contractive and unitary matrices \cite{AG1982, DKW,FoFra1984,ShYa}:

the operator
$U\in\bB(\sH\oplus\sM,\sK\oplus\sN)$
is a contraction if and only if $U$ is of the form
\begin{equation}\label{twee}
U=
\begin{bmatrix} A &D_{A^*}M \cr
KD_{A}&-KA^*M+D_{K^*}XD_{M}\end{bmatrix}: \begin{array}{l}\sH\\\oplus\\\sM\end{array}\to
\begin{array}{l}\sK\\\oplus\\\sN\end{array},
\end{equation}
where
\begin{enumerate}
\item [\rm (a)]
$A \in \bB(\sH,\sK)$, $M\in\bB(\sM,\sD_{A^*})$,
$K\in\bB(\sD_{A},\sN)$ are contractions,
\item [\rm (b)]  $X\in\bB(\sD_{M},\sD_{K^*})$ is a contraction.
\end{enumerate}
Note that
\begin{itemize}
\item if $D_A=0$, then $M=0:\sM\to\{0\}$, and $D_M=I_{\sM},$
\item if $D_{A^*}=0,$ then $K=0:\{0\}\to \sN$, $K^*=0:\sN\to\{0\}$, $D_{K^*}=I_\sN.$
\end{itemize}

The operator $U$
in \eqref{twee} is unitary if and only if (see \cite[Corollary 4.2]{AHS_CAOT})
\begin{equation}
\label{bef26a}
D_XD_M=0,\; D_KD_A=0,\;D_{X^*}D_{K^*}=0,\;D_{M^*}D_{A^*}=0.
\end{equation}
Observe that if a block-operator matrix $U$ of the form \eqref{twee} is a contraction, then due to Lemma \ref{L1} and Theorem \ref{mar14a}(4) one has
\[
D_{A^*}Mf\in\ran (I-A)\Longrightarrow||\left(K[-A^*+D_A(I-A)^{-1}D_{A^*}]M+D_{K^*}XD_M\right)f||\le ||f||
\]
and the equality is valid if $K$, $M^*$, and $X$ are isometries.
Similarly
\[
D_{A}K^*g\in\ran (I-A^*)\Longrightarrow||\left(M^*[-A+D_{A^*}(I-A^*)^{-1}D_{A}]K^*+D_{M}X^*D_{K^*}\right)g||\le ||g||
\]
and the equality is valid if $K$, $M^*$, and $X^*$ are isometries.

The next theorem describes compressions of selfadjoint extensions in terms of blocks of a unitary $2\times 2$ operator matrices acting between deficiency subspaces $\sN_i$ and $\sN_{-i}$.

\begin{theorem}\label{bef25a}
Let $S$ be a non-densely defined closed symmetric operator in $\cH$. Assume
\begin{enumerate}
\item
$\dim (\dom S)^\perp=\infty,$
\item $n_\pm'\ne 0$,
\item the operator $S$ is regular.
\end{enumerate}
Then each unitary operator $U\in\bB(\sN_i,\sN_{-i})$ such that $(U-V_i)f\ne 0$ $f\in\sL_{i}\setminus\{0\}$ takes the block $2\times 2$ operator matrix
\begin{equation}\label{bef25b}
U=
\begin{bmatrix} YV_i &D_{Y^*}M \cr
GD_{Y}V_i&-GY^*M+D_{G^*}XD_{M}\end{bmatrix}:\begin{array}{l}\sL_{i}\\ \oplus\\
\sN'_i\end{array}\to \begin{array}{l}\sL_{-i}\\ \oplus\\
\sN'_{-i}\end{array}.
\end{equation}
where
\begin{itemize}
\item $Y$ is a contraction in $\sL_{-i}$ such that $\ker (I-Y)=\{0\}$, $\dim\sD_{Y}\le n_-'$, $\dim\sD_{Y^*}\le n_+'$,
\item the operators $M^*\in\bB(\sD_{Y^*},\sN_{i}')$,
$G\in\bB(\sD_{Y},\sN_{-i}')$, are isometries and $\dim\sD_M=\dim \sD_{G^*}$,
\item $X\in\bB(\sD_{M},\sD_{G^*})$ is unitary when $\dim\sD_M=\dim \sD_{G^*}\ne 0$.
\end{itemize}
 Define the linear manifolds
\begin{equation}\label{bef26c}
\Omega_Y:=\ran (I-Y)\cap\ran D_{Y^*},\;\Omega_{Y^*}:=\ran (I-Y^*)\cap\ran D_{Y}.
\end{equation}
Then for the selfadjoint extension of $S$
\[
\left\{\begin{array}{l}
\dom \wt S=\dom S\dot+(I-U)\sN_i\\
\wt S(f_S+(I-U)\f_i)=Sf_S+i(I+U)\f_i,\;\;f_S\in\dom S,\; \f_i\in\sN_i
\end{array}\right.
\]
 the following statements are true:
\begin{enumerate}
\item the subspaces  $\ti\sN_i:=\{\varphi\in\sN_i,$ $(U-I)\varphi\in\cH_0\},$ and \\
 $\ti\sN_{-i}:=U\ti\sN_i=\{\phi\in\sN_{-i},$ $(U^{-1}-I)\phi\in\cH_0\},$ (see \eqref{ee3-34} and \eqref{mar5c}) take the form
\begin{equation}
\label{mar5b}
\begin{array}{l}
\ti\sN_i=\ker(P_{\sL_{-i}}U- V_iP_{\sL_i})=\left\{\begin{bmatrix}V_{-i}(I-Y)^{-1}h\cr\psi'_i +M^*D_{Y^*}^{[-1]}h\end{bmatrix}:h\in\Omega_Y,\; \psi'_{i}\in \ker M\right\},\\
\ti\sN_{-i}=\ker(P_{\sL_{i}}U^{-1}- V_{-i}P_{\sL_{-i}})=\left\{\begin{bmatrix} (I-Y^*)^{-1}g\cr\psi'_{-i} +GD_{Y}^{[-1]}g\end{bmatrix}:g\in\Omega_{Y^*},\; \psi'_{-i}\in \ker G^*\right\},
\end{array}
\end{equation}
where $D_{Y^*}^{[-1]}$ ($D_{Y}^{[-1]}$) is the Moore-Penrose pseudoinverse (in particular, $D_{Y^*}^{[-1]}(0)=D_{Y}^{[-1]}(0)=0$);
\item for the domain of $\wt S_0$:
\[\begin{array}{l}
\dom \wt S_0=\dom S_0\dot+(I-U_0)\dom U_0=\dom S\dot +(I-U^{-1}_0)\dom U^{-1}_0,\\[2mm]
\quad\qquad\dom U_0\subseteq\sN'_i,\; \dom U^{-1}_0\subseteq\sN'_{-i}
\end{array}
\]
the isometries $U_0$ and $U_0^{-1}$ defined in \eqref{bef8b} and \eqref{bef8bb} are of the form 
\begin{equation}\label{mar22}
\left\{\begin{array}{l}
\dom U_0=\left\{\f_i'\in\sN'_i: D_{Y^*}M\f_i'\in\Omega_Y\right\}\\[2mm]
U_0\f_i'=\left(G[-Y^*+D_Y(I-Y)^{-1}D_{Y^*}]M+D_{G^*}XD_M\right)\f_i',\;\f_{i}'\in\dom U_0
\end{array}\right.,
\end{equation}
\begin{equation}\label{mar22a}
\left\{\begin{array}{l}
\dom U^{-1}_0=\ran U_0=\left\{\f_{-i}'\in\sN'_{-i}: D_{Y}G^*\f_{-i}'\in\Omega_{Y^*}\right\}\\[2mm]
U^{-1}_0\f_{-i}'=\left(M^*[-Y+D_{Y^*}(I-Y^*)^{-1}D_{Y}]G^*+D_{M}X^*D_{G^*}\right)\f_{-i}',\\
 \f_{-i}'\in\dom U^{-1}_0
\end{array}\right.;
\end{equation}
\item for the deficiency subspaces of $\wt S_0$ the equalities
\begin{equation}\label{bef29bb}
\sN_i(\wt S_0)=\sN'_i\ominus  M^{-1}\{D_{Y^*}^{[-1]}\Omega_Y\},\; \sN_{-i}(\wt S_0)=\sN'_{-i}\ominus G^{*-1}\{D_{Y}^{[-1]}\Omega_{Y^*}\}
\end{equation}
hold.
\end{enumerate}
Moreover,
\item [\rm (I)] the following are equivalent:
\begin{enumerate}
\item [\rm (a)] $\wt S_0=S_0$, 
\item [\rm (b)] $\ker M=\{0\}$ and $\Omega_Y=\{0\}$,
\item [\rm (c)] $\ker G^*=\{0\}$ and $\Omega_{Y^*}=\{0\}$;
\end{enumerate}
\item [\rm (II)]
the following are equivalent:
\begin{enumerate}
\item [\rm (a)]the compression $\wt S_0$ is maximal symmetric with $n_+(\wt S_0)=0$ and $n_-(\wt S_0)>0$,
\item [\rm (b)] $\left\{\begin{array}{l}\ran (I-Y)=\ran (I-Y_R)^\half\\
\ran (I-Y^*)\ne\ran (I-Y_R)^\half\end{array}\right.;$
\end{enumerate}
in particular, if $Y$ is a non-unitary co-isometry ($D_Y^*=0$), then statement (a) holds;
\item [\rm (III)] the following are equivalent:
\begin{enumerate}
\item [\rm (a)]the compression $\wt S_0$ is maximal symmetric with $n_+(\wt S_0)>0$ and $n_-(\wt S_0)=0$,
\item [\rm (b)]$\left\{\begin{array}{l}\ran (I-Y^*)=\ran (I-Y_R)^\half\\
\ran (I-Y)\ne\ran (I-Y_R)^\half\end{array}\right.;$
in particular, if $Y^*$ is a non-unitary isometry ($D_Y=0$), then statement (a) holds;
\item [\rm (IV)] in the case $n_+'=n_-'$ the following are equivalent:
\begin{enumerate}
\item
the compression $\wt S_0$ is selfadjoint,
\item
$\left\{\begin{array}{l}
\ran D_{Y^*}\subseteq\ran (I-Y)\\
\ran D_{Y}\subseteq\ran (I-Y^*)
\end{array}\right..
$
\item $\ran (I-Y^*)=\ran(I-Y)= \ran (I-Y_R)^\half.$
\end{enumerate}
In particular,
\begin{itemize}
\item the selfadjoint extension $\wt S$ is regular if and only if $\ran (I-Y)=\sL_{-i}$, and if this is the case, the compression $\wt S_0$ is selfadjoint,
\item if $Y$ is unitary, then $\wt S_0$ is selfadjoint.
\end{itemize}
\end{enumerate}
\end{theorem}

\begin{proof} Note, that
in accordance with Theorem \ref{mar14a} for the linear manifolds $\Omega_Y$ and $\Omega_{Y^*}$ defined in \eqref{bef26c} the equivalence
$$\Omega_Y=\{0\}\Longleftrightarrow\Omega_{Y^*}=\{0\}$$
holds.

The block-operator matrix \eqref{bef25b} is of the form \eqref{twee} with $A=YV_i$, $D_A=V_{-i}D_YV_i$, $D_{A^*}=D_Y$, $G=KV_{-i}$ and due to properties of the operators $K,M,X$, conditions \eqref{bef26a} hold true. Moreover,
if $D_Y=0$, then in \eqref{bef25b} the operator $G=0:\{0\}\to\sN'_{-i}$ and $G^*=0:\sN'_{-i}\to\{0\}$ and $D_{G^*}=I_{\sN'_{-i}}$,
respectively $D_{Y^*}=0\Longrightarrow M^*=0:\{0\}\to\sN_{i}'$ and $M=0:\sN'_{i}\to\{0\}$ and $D_{M}=I_{\sN'_{i}}.$

It follows that $U$ is unitary, the condition $\ker(I-Y)=0$ gives that $\ker (U-V_i)\uphar\sL_i=\{0\}$.
Hence, the operator $\wt S$ defined by means of $U$ is a selfadjoint extension of $S$.

The operator $U^*=U^{-1}$ takes the form
\[
U^{-1}=
\begin{bmatrix} V_{-i}Y^* &V_{-i}D_YG^* \cr
M^*D_{Y^*}&-M^*YG^*+D_MX^*D_{G^*}\end{bmatrix}:\begin{array}{l}\sL_{-i}\\ \oplus\\
\sN'_{-i}\end{array}\to \begin{array}{l}\sL_{i}\\ \oplus\\
\sN'_{i}\end{array}.
\]
We will use Proposition \ref{bef15b}.
Then we have
\[\begin{array}{l}
\ker(P_{\sL_{-i}}U- V_iP_{\sL_i})=\left\{\begin{bmatrix} f_i\cr\f'_i \end{bmatrix}:(I-Y)V_if_i=D_{Y^*}M \f_i',\;f_i\in\sL_i,\; \f_i'\in\sN_i'\right\};\\[2mm]
\ker(P_{\sL_{i}}U^{-1}- V_{-i}P_{\sL_{-i}})=\left\{\begin{bmatrix}g_{-i}\cr\f'_{-i} \end{bmatrix}:(I-Y^*)g_{-i}=D_{Y}G^* \f_{-i}',\;g_{-i}\in\sL_{-i}. \f_{-i}'\in\sN_{-i}'\right\}.
\end{array}
\]
Since $\ran M=\sD_{Y^*}$ and $\ran G^*=\sD_Y$, from \eqref{mar5} and \eqref{march17} we arrive at \eqref{mar5b}.

From the equalities
\[
\begin{array}{l}
P_{\sN'_{i}}\ker(P_{\sL_{-i}}U- V_iP_{\sL_i})=\ker M\oplus M^*D_{Y^*}^{[-1]}\Omega_Y= M^{-1}\{D_{Y^*}^{[-1]}\Omega_Y\},\\[2mm]
P_{\sN'_{-i}}\ker(P_{\sL_{i}}U^{-1}- V_{-i}P_{\sL_{-i}})=G^{*-1}\{D_{Y}^{[-1]}\Omega_{Y^*}\},
\end{array}
\]
we get the equivalences:
\[
\begin{array}{l}
\ker(P_{\sL_{-i}}U- V_iP_{\sL_i})=\ker M\Longleftrightarrow\Omega_Y=\{0\},\\
\dom\wt S\cap\cH_0=\dom S\Longleftrightarrow \ker(P_{\sL_{-i}}U- V_iP_{\sL_i})=\{0\}\\
\Longleftrightarrow P_{\sN'_{i}}\ker(P_{\sL_{-i}}U- V_iP_{\sL_i})=\{0\}\Longleftrightarrow \left\{\begin{array}{l}\ker M=\{0\}\\
\Omega_Y=\{0\}\end{array}\right.,
\end{array}
\]
\[
\begin{array}{l}
\ker(P_{\sL_{i}}U^{-1}- V_{-i}P_{\sL_{-i}})=\ker G^*\Longleftrightarrow \Omega_{Y^*}=\{0\},\\
\dom\wt S\cap\cH_0=\dom S\Longleftrightarrow \ker(P_{\sL_{i}}U^{-1}- V_{-i}P_{\sL_{-i}})=\{0\}\\
\Longleftrightarrow P_{\sN'_{-i}}\ker(P_{\sL_{i}}U^{-1}- V_{-i}P_{\sL_{-i}})=\{0\}\Longleftrightarrow \left\{\begin{array}{l}\ker G^*=\{0\}\\
\Omega_{Y^*}=\{0\}\end{array}\right..
\end{array}
\]
Since $M(\psi_i+M^*D_{Y^*}^{[-1]}h)=D_{Y^*}^{[-1]}h$ for $\psi'_i\in\ker M$, $h\in\Omega_Y$, and $G^*(\psi_{-i}+GD_{Y}^{[-1]}g)=D_{Y}^{[-1]}g$ for $\psi'_{-i}\in\ker G^*$, $g\in\Omega_{Y^*}$,
 expressions for $U_0$, $U^{-1}_0$ in \eqref{mar22} and \eqref{mar22a} follow from \eqref{bef8}, \eqref{mar5b}, and Remark \ref{mar23a}.

Now the representations in \eqref{bef8} imply that the deficiency subspaces of the operator $\wt S_0$ are of the form \eqref{bef29bb}.
Hence
\[
n_{+}(\wt S_0)=\dim (\sN_i'\ominus M^{-1}\{D_{Y^*}^{[-1]}\Omega_Y\}),\; n_-(\wt S_0)=\dim \left(\sN_{-i}'\ominus G^{*-1}\{D_{Y}^{[-1]}\Omega_{Y^*}\}\right).
\]
Besides
\begin{equation}
\label{mar15bb}
\begin{array}{l}
\Omega_{Y}=\{0\}\Longleftrightarrow \Omega_{Y^*}=\{0\}\Longleftrightarrow M^{-1}\{D_{Y^*}^{[-1]}\Omega_{Y}\}=\ker M\\[2mm]
\Longleftrightarrow G^{*-1}\{D_{Y}^{[-1]}\Omega_{Y^*}\}=\ker G^*\Longrightarrow\left\{\begin{array}{l}n_+(\wt S_0)=\dim\sD_{Y^*}\\[2mm]
n_-({\wt S_0})=\dim\sD_{Y}\end{array}\right.,\\[2mm]
\Omega_Y=\ran D_{Y^*}\Longleftrightarrow M^{-1}\{D_{Y^*}^{[-1]}\Omega_{Y}\}=M^{-1}\{\sD_{Y^*}\}=\sN'_i\Longleftrightarrow n_{+}(\wt S_0)=0,\\[2mm]
\Omega_{Y^*}=\ran D_{Y}\Longleftrightarrow G^{*-1}\{D_{Y}^{[-1]}\Omega_{Y^*}\}=G^{*-1}\{\sD_{Y}\}=\sN'_{-i}\Longleftrightarrow n_{-}(\wt S_0)=0.
\end{array}
\end{equation}
The equivalences in (4)--(7) follow from \eqref{mar5b}, 
\eqref{mar15bb}, Theorem \ref{mar14a} and Remark \ref{apr26aa}.

Suppose $\ran (I-Y)=\sL_{-i}.$ Then $\ran D_{Y^*}\subset \ran(I-Y)$ and $\ran D_{Y}\subset \ran(I-Y^*)$ and $\wt S_0$ is selfadjoint in $\cH_0$. Besides, because for $f_i\in\sL_{i}$ from \eqref{bef25b}
\[
||(U-V_i)f_i||^2=||(Y-I)V_if_i||^2+||D_YV_i f_i||^2\ge c||V_if_i||^2=c||f_i||^2
\]
 for some $c>0$, we get that $\ran (U-V_i)\sL_i$ is a subspace in $\sN_{-i}$. Due to Remark \ref{mar9bb}, the selfadjoint extension $\wt S$ is regular.

Conversely, suppose $\wt S$ is a regular selfadjoint extension of $S$, then $\ran (U-V_i)\sL_i$ is a subspace in $\sN_{-i}$. Hence
\[
||(U-V_i)f_i||^2=||(Y-I)V_if_i||^2+||D_YV_i f_i||^2\ge c||f_i||^2 \;\;\forall f_i\in\sL_i
\]
with some $c>0$.
Since
\[
\begin{array}{l}
||(I-Y)h||^2+||D_Yh||^2=2||h||^2-2\RE(Yh,h)=2\RE((I- Y)h,h)\le 2|((I- Y)h,h)|\\[2mm]
\le 2||h||||(I-Y)h||\;\;\forall h\in\sL_{-i},
\end{array}
\]
the inequality $||(I-Y)h||^2+||D_Yh||^2\ge c||h||^2$
implies
\[
||(I-Y)h||\ge \half c||h||\;\;\forall h\in\sL_{-i}.
\]
Hence $\ran (I-Y)=\sL_{-i}$.
\end{proof}
\begin{remark} \label{mar29a}
Due to Theorem \ref{mar14a} and Lemma \ref{L1} the operators $U_0$ and $U^{-1}_0$ defined in \eqref{mar22} and \eqref{mar22a}
are isometric on $\dom U_0$ and $\dom U^{-1}_0$, respectively.
\end{remark}
\begin{theorem}\label{bef29ab}
Let $S$ be a non-densely defined closed symmetric operator in $\cH$ and assume
\begin{enumerate}
\item
$\dim (\dom S)^\perp=\infty,$
\item the operator $S$ is regular.
\end{enumerate}
Let $\wt S_0=S_0$ or let $\wt S_0$ be a symmetric extension of $S_0$. Then  there exists a selfadjoint extension $\wt S$ of $S$ in $\cH$ such that the compression
$P_{\cH_0}\wt S\uphar\dom (\wt S\cap\cH_0)$ coincides with $\wt S_0$.
\end{theorem}
\begin{proof}
We will use Theorem \ref{bef25a}, expressions in \eqref{mar22}, \eqref{mar22a}, and Proposition \ref{apr7a}.

For the domain $\dom\wt S_0$ one has
\[
\dom \wt S_0=\dom S\dot+(I-\wt U_0)\dom \wt U_0,\;\dom \wt U_0\subseteq \sN_i',
\]
where $\wt U_0$ maps $\dom \wt U_0$ isometrically into $\sN_{-i}'$.
If $\wt S_0=S_0$, then $\dom \wt U_0=\{0\}$.

Set
\[
\wt \sN_i':=\sN_i'\ominus\dom\wt U_0,\; \wt \sN_{-i}':=\sN_{-i}'\ominus\ran \wt U_0.
\]
Note that if $\wt S_0$ is a maximal symmetric extension of $S_0$, then $\dom \wt U_0=\sN_i'$, $\wt \sN_i'=\{0\}$ when $n_+'\le n_-'$
and $\ran \wt U_0=\sN'_{-i}$, $\wt \sN_{-i}'=\{0\}$ when $n_+'\ge n_-'.$

By Proposition \ref{apr7a} there exists in $\sL_{-i}$ a contraction $Y$ such that
$$\ker (I-Y)=\{0\}, \;\Omega_Y=\{0\},\;\dim \sD_{Y}=\dim \wt \sN_i',\; \dim \sD_{Y^*}=\wt \sN_{-i}'.$$
Further construct a unitary $U$ of the form \eqref{bef25b} with isometry $G\in\bB(\sD_{Y},\sN_{-i}')$, co-isometry $M\in\bB(\sN_{i}',\sD_{Y^*})$ 
such that
\[
\ran M^*=\wt\sN_i'\;(\Longleftrightarrow\ker M=\dom\wt U_0),\;\ran G=\wt\sN'_{-i}\;(\Longleftrightarrow\ker G^*= \ran\wt U_0).\\
\]
Then $\sD_M=\dom \wt U_0$, $\sD_{G^*}=\ran \wt U_0$. Set $X:=\wt U_0\in\bB(\sD_M,\sD_{G^*})$.

From \eqref{mar22} we get that
$
U_0=D_{G^*}XD_{M}=\wt U_0.
$
Then for the selfadjoint extension $\wt S$ of $S$ in $\cH$ with
$\dom \wt S=\dom S\dot+(I-U)\sN_i$
we get that the compression of $\wt S$ coincides with given $\wt S_0.$

 In particular,
\begin{enumerate}
\item
if $\wt S_0=S_0$, then $U$ is of the form
$
U=\begin{bmatrix} YV_i &D_{Y^*}M \cr
GD_{Y}V_i&-GY^*M\end{bmatrix}:\begin{array}{l}\sL_i\\ \oplus\\
\sN_i'\end{array}\to \begin{array}{l}\sL_{-i}\\ \oplus\\
\sN_{-i}'\end{array},
$
where $M\in\bB(\sN_i',\sD_{Y^*})$ and $G\in\bB(\sD_{Y},\sN_{-i}')$ are unitary operators; it follows from \eqref{mar22} that $\dom U_0=\{0\}$;
\item
 if (a) $\wt S_0$ is a maximal symmetric extension of $S_0$ with $\dom \wt U_0=\sN'_i$ (the case $n_+'\le n_-')$, then take a co-isometry $Y$ ($D_{Y^*}=0$); hence
$M=0:\sN_{i}'\to\{0\}$, $D_{M}=I_{\sN_i'}$, $D_{G^*}=P_{\ran \wt U_0}$, i.e.,
$
U=\begin{bmatrix} YV_i &0 \cr
GD_{Y}V_i& \wt U_0\end{bmatrix}:\begin{array}{l}\sL_i\\ \oplus\\
\sN_i'\end{array}\to \begin{array}{l}\sL_{-i}\\ \oplus\\
\sN_{-i}'\end{array};
$

(b) $\wt S_0$ is a maximal symmetric extension of $S_0$ with $\ran\wt U_0=\sN'_{-i}$ (the case $n_-'\le n_+')$, then $Y$ is a non-unitary isometry and
$ 
U=\begin{bmatrix} YV_i &D_{Y^*}M \cr
0& \wt U_0P_{\dom \wt U_0}\end{bmatrix}:\begin{array}{l}\sL_i\\ \oplus\\
\sN_i'\end{array}\to \begin{array}{l}\sL_{-i}\\ \oplus\\
\sN_{-i}'\end{array};
$
\item
if $\wt S_0$ is a selfadjoint extension of $S_0$ ($n_+'=n_-'$), then $\dom U_0=\sN_i'$, $\ran U_0=\sN'_{-i}$, $Y$ an arbitrary unitary in $\sL_{-i}$ with $\ker (I-Y)=\{0\}$,
$
U=\begin{bmatrix} YV_i &0 \cr
0&\wt U_0\end{bmatrix}:\begin{array}{l}\sL_{i}\\ \oplus\\
\sN'_i\end{array}\to \begin{array}{l}\sL_{-i}\\ \oplus\\
\sN'_{-i}\end{array}.
$
\end{enumerate}
\end{proof}

\section{Compressions of maximal dissipative extensions}\label{apr11c}

Let $S$ be a closed non-densely defined symmetric operator in $\cH$ and, as before, $\cH_0:=\cdom S$, $S_0:=P_{\cH_0}S.$ If $\wt S$ is a maximal dissipative extension of $S$ in $\cH$, then its adjoint $\wt S^*$ is maximal accumulative extension of $S$ i.e., $-\wt S^*$ is maximal dissipative extension of $-S$.
Set
$$\wt S_{*0}:=P_{\cH_0}\wt S^*\uphar(\dom\wt S^*\cap\cH_0).$$
Clearly $\wt S_0$ and $-\wt S_{*0}$ are dissipative operators,
$
S_0\subseteq \wt S_0,\;S_0\subseteq \wt S_{*0}
$
and
\begin{equation}\label{jul14a}
(\wt S_0f,g)=(f, \wt S_{*0}g)\;\;\forall f\in\dom \wt S_0,\;\forall g\in\dom \wt S_{*0}.
\end{equation}
Moreover, the following are equivalent:
\begin{enumerate}
\def\labelenumi{\rm (\roman{enumi})}
\item the operators $\wt S_0$ and $-\wt S_{*0}$ are maximal dissipative in $\cH_0$;
\item $(\wt S_0)^*=\wt S_{*0};$
\end{enumerate}
If $\codim \cH_0<\infty$, then according to \cite{Nud} the equality $\wt S_{*0}=S^*_0$  holds.

If $n_+'=0$ ($n_-'=0$), then  $\wt S_0=S_0$ (respectively, $\wt S_{*0}=S_0$).

Let us formulate an analogue of Proposition \ref{bef15b}.
\begin{proposition}
\label{mar11a}
Let $S$  be a non-densely defined regular symmetric operator $S$ 
and let $\wt S$ be a maximal dissipative
 extension of $S$,
\[
\dom \wt S=\dom S\dot+(I-U)\sN_i,
\]
where $U\in\bB(\sN_i,\sN_{-i})$ is a contraction and $\ker (U-V_i)\uphar\sL_i=\{0\}.$
Then
\begin{enumerate}
\item the equalities
\[
\left\{\begin{array}{l}
\dom \wt S\cap \cH_0=\dom S\dot+(I-U)\ker(P_{\sL_{-i}}U- V_iP_{\sL_i})\\[2mm]
\qquad\qquad =\dom S\dot+ (I-P_{\sN'_{-i}}U)P_{\sN'_i}\ker(P_{\sL_{-i}}U- V_iP_{\sL_i}),\\[2mm]
\dom \wt S^*\cap \cH_0= \dom S\dot + (I-U^{*})\ker (V_{-i}P_{\sL_{-i}}-P_{\sL_{i}}U^{*})\\[2mm]
=\dom S\dot+ (I-P_{\sN'_{i}}U^{*})P_{\sN'_{-i}}\ker(V_{-i} P_{\sL_{-i}}-P_{\sL_{i}}U^{*})
\end{array}\right.
\]
hold;
 \item the operators $U_0$ and $U_{*0}$ defined as follows
\begin{equation}\label{may5a}
\left\{\begin{array}{l}
 U_0:P_{\sN_i'}\f\mapsto P_{\sN'_{-i}}U\f,\;\;\f\in \ker (P_{\sL_{-i}}U- V_iP_{\sL_i}),\\
U_{*0}:P_{\sN_{-i}'}\psi\mapsto P_{\sN'_{i}}U^*\psi,\;\;\psi\in \ker (P_{\sL_{i}}U^*- V_{-i}P_{\sL_{-i}})
\end{array}\right.
\end{equation}
are contractions and
the following are valid:
\begin{equation}\label{jul14bb}
\begin{array}{l}
\left\{\begin{array}{l}\dom \wt S_0=\dom S\dot+ (I-U_0)\dom U_0,\\
\wt S_0(f_S+(I-U_0)\f_i')=S_0f_S+i(I+U_0)\f_i',\; f_S\in\dom S,\;\f_i'\in\dom U_0,
\end{array}\right.\\[2mm]
\left\{\begin{array}{l}\; \dom \wt S_{*0}=\dom S\dot+ (I-U_{*0})\dom U_{*0},\\
\wt S_{*0}(g_S+(I-U_{*0})\f_{-i}')=S_0g_S-i(I+U_{*0})\f_{-i}',\; g_S\in\dom S,\;\f_{-i}'\in\dom U_{*0}\end{array}\right.,
\end{array}
\end{equation}
\begin{equation}
\label{juni28ff}
(U_0\f_{i}, \psi_{-i})=(\f_{i}, U_{*0}\psi_{-i})\;\;\forall \f_{i}\in\dom U_0,\;\forall \psi_{-i}\in\dom U_{*0};
\end{equation}
\item the following are equivalent:
\begin{enumerate}
\def\labelenumi{\rm (\roman{enumi})}
\item $\dom\wt S\cap\cH_0=\dom S,$
\item $\ker(P_{\sL_{-i}}U- V_iP_{\sL_i})=\{0\},$
\item $P_{\sN'_i}\ker(P_{\sL_{-i}}U- V_iP_{\sL_i})=\{0\}$;
\end{enumerate}
\item the following are equivalent:
\begin{enumerate}
\item $\dom\wt S^*\cap\cH_0=\dom S,$
\item $\ker (V_{-i} P_{\sL_{-i}}-P_{\sL_{i}}U^{*})=\{0\}$,
\item $P_{\sN'_{-i}}\ker(V_{-i} P_{\sL_{-i}}-P_{\sL_{i}}U^{*})=\{0\};$
\end{enumerate}
in particular,
if one of the semi-deficiency subspaces is trivial, then $\dom \wt S\cap \cH_0=\dom S$ and the operator
$\wt S_0:=P_{\cH_0}\wt S\uphar(\dom \wt S\cap \cH_0)$ is maximal symmetric;
\item
the following are equivalent:
\begin{enumerate}
\def\labelenumi{\rm (\roman{enumi})}
\item operator $\wt S_0$ is maximal dissipative in $\cH_0,$
\item $P_{\sN'_i}\ker(P_{\sL_{-i}}U- V_iP_{\sL_i})=\sN'_i$; 
\end{enumerate}
\item
the following are equivalent:
\begin{enumerate}
\def\labelenumi{\rm (\roman{enumi})}
\item the operator $-\wt S_{*0}$ is maximal dissipative in $\cH_0,$
\item $P_{\sN'_{-i}}\ker(V_{-i} P_{\sL_{-i}}-P_{\sL_{i}}U^{*})=\sN'_{-i}.$
\end{enumerate}
\item
the following are equivalent:
\begin{enumerate}
\def\labelenumi{\rm (\roman{enumi})}
\item $\wt S_{*0}=(\wt S_0)^*$;
\item $\left\{\begin{array}{l} P_{\sN'_i}\ker(P_{\sL_{-i}}U- V_iP_{\sL_i})=\sN'_i\\
P_{\sN'_{-i}}\ker(V_{-i} P_{\sL_{-i}}-P_{\sL_{i}}U^{*})=\sN'_{-i}\end{array}\right.;$
\end{enumerate}
\item if $\codim\cH_0<\infty$, then $\wt S_{*0}=(\wt S_0)^*.$
\end{enumerate}
\end{proposition}
\begin{proof}
Taking into account that $U$ is a contraction and $\ker(U-V_i)\uphar\sL_i=\{0\}$, the proof is similar to the proof of Proposition \ref{bef15b}.
Besides the statement that \textit{the operator $T$ is maximal dissipative $\Longleftrightarrow$ $T$ is densely defined and dissipative and $-T^*$ is dissipative} has to be used.
The equality \eqref{juni28ff} follows from \eqref{jul14a} and  \eqref{jul14bb}.

Suppose $\codim\cH_0<\infty$. Since the subspaces $\sL_\lambda$ are finite-dimensional, $V_\lambda$ unitarily  maps $\sL_\lambda$ onto $\sL_{\bar\lambda},$
$U\in\bB(\sN_i,\sN_{-i})$ is an admissible contraction, the subspaces $\ker(P_{\sL_{-i}}U- V_iP_{\sL_i})$ and $ \ker(V_{-i} P_{\sL_{-i}}-P_{\sL_{i}}U^{*})$
take the form
\[
\begin{array}{l}
\ker(P_{\sL_{-i}}U- V_iP_{\sL_i})=\left\{\begin{bmatrix}(V_i-P_{\sL_{-i}}U\uphar\sL_i)^{-1}P_{\sL_{-i}}U\f'_i\cr\f_i'\end{bmatrix}:\; \f'_i\in\sN_i'   \right\},\\
\ker(V_{-i} P_{\sL_{-i}}-P_{\sL_{i}}U^{*})=\left\{\begin{bmatrix}(V_{-i}-P_{\sL_{i}}U^*\uphar\sL_{-i})^{-1}P_{\sL_{i}}U^*\f'_{-i}\cr\f_{-i}'\end{bmatrix}:\; \f'_{-i}\in\sN_{-i}'   \right\}.
\end{array}
\]
Hence $P_{\sN'_i}\ker(P_{\sL_{-i}}U- V_iP_{\sL_i})=\sN'_i$ and $P_{\sN'_{-i}}\ker(V_{-i} P_{\sL_{-i}}-P_{\sL_{i}}U^{*})=\sN'_{-i}.$
Consequently, $\wt S_{*0}=(\wt S_0)^*.$
\end{proof}
\section{Parametrization of admissible contractive  $2\times 2$ block operator matrices and compressions of maximal dissipative extensions}
An admissible contraction $U\in\bB(\sN_i,\sN_{-i})$ can be represented by operator-matrix as follows:
\begin{enumerate}
\item if $n_\pm'\ne 0$, then
\begin{equation}\label{mar11bbb}
U=
\begin{bmatrix} YV_i &D_{Y^*}M \cr
GD_{Y}V_i&-GY^*M+D_{G^*}XD_{M}\end{bmatrix}:\begin{array}{l}\sL_{i}\\ \oplus\\
\sN'_i\end{array}\to \begin{array}{l}\sL_{-i}\\ \oplus\\
\sN'_{-i}\end{array}.
\end{equation}
\item if $n_+'=0$ and $n_-'\ne 0,$ then
\begin{equation}\label{11bbbmar}
U=
\begin{bmatrix} YV_i \cr
GD_{Y}V_i\end{bmatrix}:\sN_i=\sL_{i}\to \begin{array}{l}\sL_{-i}\\ \oplus\\
\sN'_{-i}\end{array},
\end{equation}
\item if $n_-=0$ and $n_+\ne 0$, then
\begin{equation}\label{bbbmar12}
U=
\begin{bmatrix} YV_i &D_{Y^*}M \end{bmatrix}:\begin{array}{l}\sL_{i}\\ \oplus\\
\sN'_i\end{array}\to \sL_{-i}=\sN_{-i},
\end{equation}
\end{enumerate}
where\begin{equation}\label{mar11bbbb}
\left\{\begin{array}{l}
Y \;\mbox{is a contraction in}\; \sL_{-i},\;\ker (I-Y)=\{0\},\\
\mbox{the operators}\; M\in\bB(\sN_{i}',\sD_{Y^*}),\;G\in\bB(\sD_{Y},\sN_{-i}')\;
\mbox{are contractions}\\
 X\in\bB(\sD_{M},\sD_{G^*})\;\mbox{is a contraction}.
\end{array}\right.
\end{equation}

\begin{theorem}\label{mar11aaa}
Let $S$ be a non-densely defined closed regular symmetric operator and $n_\pm'\ne 0$. Let $U\in\bB(\sN_i,\sN_{-i})$ be
an admissible contraction of the form \eqref{mar11bbb}. 
  Assume
$\dim (\dom S)^\perp=\infty.$
Then for the maximal dissipative extension of $S$
\[
\left\{\begin{array}{l}
\dom \wt S=\dom S\dot+(I-U)\sN_i,\\
\wt S(f_S+(I-U)\f_i)=Sf_S+i(I+U)\f_i,\;\;f_S\in\dom S,\; \f_i\in\sN_i
\end{array}\right.,
\]
 the following assertions are valid:
\begin{enumerate}
\item $\ker (P_{\sL_{-i}}U- V_iP_{\sL_i})=\left\{\begin{bmatrix} f_i\cr\f'_i \end{bmatrix}:(I-Y)V_if_i=D_{Y^*}M \f_i',\;f_i\in\sL_i,\; \f_i'\in\sN_i'\right\}, $\\
$\ker(P_{\sL_{i}}U^{*}- V_{-i}P_{\sL_{-i}})=\left\{\begin{bmatrix}g_{-i}\cr\f'_{-i} \end{bmatrix}:(I-Y^*)g_{-i}=D_{Y}G^* \f_{-i}',\;g_{-i}\in\sL_{-i},\; \f_{-i}'\in\sN_{-i}'\right\}$;
\item for the operators $U_0$ and $U_{*0}$ defined in \eqref{may5a} the equalities
\begin{equation}\label{mar222}
\begin{array}{l}
\left\{\begin{array}{l}
\dom U_0=\left\{\f_i'\in\sN'_i: D_{Y^*}M\f_i'\in\ran (I-Y)\right\},\\[2mm]
U_0\f_i'=\left(G[-Y^*+D_Y(I-Y)^{-1}D_{Y^*}]M+D_{G^*}XD_M\right)\f_i',\;\f_{i}'\in\dom U_0
\end{array}\right.\end{array},
\end{equation}
\begin{equation}\label{mar222a}
\left\{\begin{array}{l}
\dom U_{*0}=\left\{\f_{-i}'\in\sN'_{-i}: D_{Y}G^*\f_{-i}'\in\ran(I-Y^*)\right\},\\[2mm]
U_{*0}\f_{-i}'=\left(M^*[-Y+D_{Y^*}(I-Y^*)^{-1}D_{Y}]G^*+D_{M}X^*D_{G^*}\right)\f_{-i}',\;
 \f_{-i}'\in\dom U_{*0}
\end{array}\right.
\end{equation}
hold.
\end{enumerate}
Moreover,
\begin{enumerate}
\item [\rm (I)] the following are equivalent:
\begin{enumerate}
\item $\wt S_0=S_0$,
\item $\ker M=\{0\}$ and $\ran (I-Y)\cap D_{Y^*}(\ran M)=\{0\}$,
\end{enumerate}
\item [\rm (II)]the following are equivalent:
\begin{enumerate}
\item $\wt S_{*0}=S_0$,
\item $\ker G^*=\{0\}$ and $\ran (I-Y^*)\cap D_{Y}(\ran G^*)=\{0\}$;
\end{enumerate}
\item [\rm (III)] the compression $\wt S_0$ is maximal dissipative if and only if
$$\ran (I-Y)\supset D_{Y^*}(\ran M);$$
in particular, if $\ran (I-Y)=\ran (I-Y_R)^\half$, then $\wt S_0$ is maximal dissipative;
\item [\rm (IV)]the operator $-\wt S_{*0}$ is maximal dissipative if and only if
\[
\ran (I-Y^*)\supset D_{Y}(\ran G^*);
\]
in particular, if $\ran (I-Y^*)=\ran (I-Y_R)^\half$, then $-\wt S_{*0}$ is maximal dissipative;
\item [\rm (V)]the following are equivalent
\begin{enumerate}
\def\labelenumi{\rm (\roman{enumi})}
\item $\wt S_{*0}=(\wt S_0)^*$;
\item $\left\{\begin{array}{l} \ran (I-Y)\supset D_{Y^*}(\ran M)\\
\ran (I-Y^*)\supset D_{Y}(\ran G^*)\end{array}\right.;$
\end{enumerate}
in particular, if $\ran (I-Y)=\ran (I-Y^*)=\ran (I-Y_R)^\half$, then $\wt S_{*0}=(\wt S_0)^*$;
\item [\rm (VI)]the following are equivalent:
\begin{enumerate}
\item the operator $\wt S$ is regular;
\item the operator $\wt S^*$ is regular;
\item $\ran (I-Y)=\sL_{-i}$.
\end{enumerate}
If one of the conditions (a)--(c) is fulfilled, then $\wt S_0$ and $-\wt S_{*0}$ are maximal dissipative in $\cH_0$.
\end{enumerate}
\end{theorem}
\begin{theorem}\label{gua31}
Let $S$ be a non-densely defined closed regular symmetric operator, $\dim (\dom S)^\perp=\infty.$
\begin{enumerate}
\item
If $n_+'=0,$ and $n_-'\ne 0 $, then for the maximal dissipative extension of $S$
\[
\left\{\begin{array}{l}
\dom \wt S=\dom S\dot+(I-U)\sN_i,\\
\wt S(f_S+(I-U)\f_i)=Sf_S+i(I+U)\f_i,\;\;f_S\in\dom S,\; \f_i\in\sN_i
\end{array}\right.,
\]
where $U$ is of the form \eqref{11bbbmar}, the following assertions hold true:
\begin{enumerate}
\item $\wt S_0=S_0$;
\item $ S_{*0}= S^*_0\uphar\left(\dom S\dot+\left\{\f_{-i}'\in\sN_{-i}':D_YG^*\f_{-i}'\in\ran(I-Y^*)\right\}\right).$

In particular, if $\ran (I-Y^*)\cap\ran D_{Y}=\{0\}$, then $\wt S_{*0}=S^*_0\uphar\left(\dom S\dot+\ker G^*\right).$
\end{enumerate}
\item If $n_+'\ne 0,$ and $n_-'= 0 $, then for the maximal dissipative extension of $S$
\[
\left\{\begin{array}{l}
\dom \wt S=\dom S\dot+(I-U)\sN_i,\\
\wt S(f_S+(I-U)\f_i)=Sf_S+i(I+U)\f_i,\;\;f_S\in\dom S,\; \f_i\in\sN_i
\end{array}\right.,
\]
where $U$ is of the form \eqref{bbbmar12}, the following assertions hold true:
\begin{enumerate}
\item $\wt S_{*0}=S_0$;
\item $ \wt S_{0}= S^*_0\uphar\left(\dom S\dot+\left\{\f_{i}'\in\sN_{i}':D_{Y^*}M\f_{i}'\in\ran(I-Y)\right\}\right).$

In particular, if $\ran (I-Y^*)\cap\ran D_{Y}=\{0\}$, then $\wt S_{0}=S^*_0\uphar\left(\dom S\dot+\ker M\right).$
\end{enumerate}
\end{enumerate}
\end{theorem}

The proof of Theorem \ref{mar11aaa} and Theorem \ref{gua31}  are similar to the proof of Theorem \ref{bef25a}.
Note, that another proof of the statement that the compression $\wt S_0$ is maximal dissipative in $\cH_0$ when $\wt S$ is a regular maximal dissipative extension of $S$ is given in \cite[Theorem 4.1.10]{ArlBelTsek2011}.

\begin{theorem}\label{juni27bb} Suppose $\dim(\dom S)^\perp=\infty$ and $n_\pm '\ne 0$.

{\rm(1)}
Let $U\in\bB(\sN_i,\sN_{-i})$ be a contraction of the form \eqref{mar11bbb} whose entries have the properties
\begin{enumerate}
\item [{\rm (a)}] $Y\in\bB(\sL_{-i})$ is a contraction and $\ran (I-Y)\cap\ran D_{Y^*}=\{0\}$,
\item [{\rm (b)}] $M\in\bB(\sN'_i,\sD_{Y^*})$ and $G\in\bB(\sD_{Y},\sN'_{-i})$ are non-zero contractions,
\begin{equation}\label{juni27d}
\ker M\ne\{0\},\; \ker G^*\ne\{0\}.
\end{equation}
\end{enumerate}
Then for any contraction $X\in\bB(\sD_M,\sD_{G^*})$, for the corresponding to $U$ the maximal dissipative extension $\wt S$,
$
\dom\wt S=\dom S\dot+(I-U)\sN_i
$
and for contractions $U_0$ and $U_{*0}$ defined in \eqref{mar222} and \eqref{mar222a} one has
\begin{equation}\label{juni27cc}
\begin{array}{l}
\dom U_0=\ker M,\; U_0=D_{G^*}X\uphar\ker M,\\[2mm]
\dom U_{*0}=\ker G^*,\; U_{*0}=D_{M}X^*\uphar\ker G^*,
\end{array}
\end{equation}
\begin{equation}\label{juni27e}
\begin{array}{l}
U_0h\notin \dom U_{*0}\Longrightarrow ||U_0 h||<||h||,\\[2mm]
U_{*0}f\notin \dom U_{0}\Longrightarrow ||U_{*0} f||<||f||.
\end{array}
\end{equation}
In terms of the corresponding compressions $\wt S_0$ and $\wt S_{*0}$ conditions \eqref{juni27e} are equivalent to the following equalities
\begin{equation}\label{juli24a}
\begin{array}{l}
\dom\wt S_0\cap\dom \wt S_{*0}=\{f\in\dom \wt S_0:\IM(\wt S_0f, f)=0\}\\[2mm] 
=
\{g\in\dom\wt S_{*0}:\IM(\wt S_{*0}g,g)=0\}. 
\end{array}
\end{equation}

{\rm(2)} If
$\ker M\ne\{0\},\; \ker G^*=\{0\},$ ($\ker G^*\ne\{0\},\; \ker M=\{0\}$),
then for any contraction $X\in\bB(\sD_M,\sD_{G^*})$,  for the corresponding to $U$ the maximal dissipative extension $\wt S$
and for contractions defined in \eqref{mar222} and \eqref{mar222a} one has
\[
\begin{array}{l}
\dom U_0=\ker M,\; U_0=D_{G^*}X\uphar\ker M,\;\;\dom U_{*0}=\{0\}\\[2mm]
(\mbox{respectively},\; \dom U_{0}=\{0\},\;\;\dom U_{*0}=\ker G^*,\; U_{*0}=D_M X^*\uphar\ker G^*),
\end{array}
\]
and
\begin{equation}\label{juni28e}
\begin{array}{l}
||U_0 h||<||h||\;\;\forall h\in\dom U_0\setminus\{0\},\\[2mm]
(\mbox{respectively},\; ||U_{*0} g||<||g||\;\;\forall g\in\dom U_{*0}\setminus\{0\}).
\end{array}
\end{equation}
\end{theorem}
\begin{proof}

(1) Due to \eqref{juni27d} we get
\[
\begin{array}{l}
D_{M}\uphar\ker M=I_{\ker M},\; ||D_M h||<||h||\;\forall h\in\cran M^*\setminus\{0\},\\[2mm]
D_{G^*}\uphar\ker G^*=I_{\ker G^*}, \; ||D_{G^*} f||<||f||\;\forall f\in\cran G\setminus\{0\}.
\end{array}
\]
Condition (a) on $Y$, the equivalence (2) in Theorem \ref{mar14a}, and  \eqref{mar222}, \eqref{mar222a} yield \eqref{juni27cc}.

Since ${\ker G^*}=\dom U_{*0}$ and for $h\in\dom U_0$ we have
$$U_0h=P_{\ker G^*}Xh+D_{G^*}P_{\cran G}Xh,$$
the condition $U_0h\notin \dom U_{*0}$ implies $P_{\cran G}Xh\ne 0$. Hence
\[
\begin{array}{l}
||U_0 h||^2=||P_{\ker G^*}Xh||^2+||D_{G^*}P_{\cran G}Xh||^2\\[2mm]
< ||P_{\ker G^*}Xh||^2+||P_{\cran G}Xh||^2=||Xh||^2\le ||h||^2.
\end{array}
\]
Thus, we get the first implication in \eqref{juni27e}, the second can be proved similarly.

The relation \eqref{juni28ff} follows from \eqref{juni27cc}.

Let us prove the equivalence of \eqref{juni27e} and \eqref{juli24a}. If $f\in\dom \wt S_0\cap\dom\wt S_{*0}$, then because $\wt S_0 f=S_0^*f$, $\wt S_{*0}f=S_0^*f$ and $\wt S_0$ and
$-\wt S_{*0}$ are dissipative operators, we have
$$0\le\IM(\wt S_0 f,f)=\IM(S^*_0f,f),\; 0\ge \IM(\wt S_{*0}f,f)=\IM(S^*_0f,f).$$
Consequently, $\IM(\wt S_0 f,f)=\IM(\wt S_{*0} f,f)=0.$

Note that from $\wt S_0 (f_S+(I-U_0)h)= S_0f_S+i(I+U_0)h$, $f_S\in\dom S,\; h\in\dom U_0$ it follows the equality
\begin{equation}\label{31ajul}
\IM \left(S_0f_S+i(I+U_0)h,f_S+(I-U_0)h\right)=||h||^2-||U_0h||^2.
\end{equation}
 Suppose \eqref{juni27e} and assume for $f=(I-U_0)h(\in\dom \wt S_0)$ that the equality $\IM(\wt S_0 f,f)=0$ holds. Then from \eqref{31ajul}
we get $||h||^2=||U_0h||^2$. The equality \eqref{juni27e} gives $U_0h\in\dom U_{*0}$.
Now from \eqref{juni28ff} with $\f_i=h,$ $\psi_{-i}=U_0h$:
 \[
(U_0h, U_{0}h)=(h, U_{*0}U_0h).
\]
Hence by the Cauchy-Schwartz inequality
\[
||h||^2=||U_0 h||^2=(h, U_{*0}U_0h)=|(h, U_{*0}U_0h)|\le ||h||||U_{*0}U_0 h|| \Longrightarrow ||h||\le ||U_{*0}U_0h||.
\]
Because the operator $U_{*0}U_0$ is a contraction, we obtain $||U_{*0}U_0h||=||h||$ and $(h, U_{*0}U_0h)=||h||||U_{*0}U_0h||.$
It follows that $U_{*0}U_0h=h.$
The latter gives that
\[
f=(I-U_0)h=((I-U_{*0})(-U_0h)\in\dom \wt S_{*0}.
\]
Thus, $f\in\dom\wt S_0\cap\wt S_{*0}.$

Now suppose the implication $\IM(\wt S_0f, f)=0 \Longrightarrow f\in \dom\wt S_0\cap\dom \wt S_{*0}$ holds.
Assume $||\f_i||=||U_0\f_i||$ for some $\f_i\in\dom U_0$. Then for $f=(I-U_0)\f_i\in\dom \wt S_0$  holds the equality $\IM(\wt S_0f, f)=0$, which implies
$f\in\dom \wt S_{*0}$, i.e.,
there is $\psi_{-i}\in\dom U_{*0}$ such that $f=(U_{*0}-I)\psi_{-i}$. Hence
\[
(I-U_0)\f_i=(U_{*0}-I)\psi_{-i}\Longleftrightarrow U_0\f_i=\psi_{-i} \quad\mbox{and}\quad \f_i=U_{*0}\psi_{-i}.
\]
Thus, $U_0\f_i\in\dom U_{*0}$. Therefore, if $U_0h\notin\dom U_{*0}$, then necessary $||U_0h||<||h||.$

Similarly the implication $\IM(\wt S_{*0}g, )=0 \Longrightarrow g\in \dom\wt S_0\cap\dom \wt S_{*0}$ implies that
if $U_{*0}h\notin\dom U_{0}$, then necessary $||U_{*0}h||<||h||.$

Thus, \eqref{juli24a} $\Longrightarrow$ \eqref{juni27e}.

(2) Consider the case $\ker M\ne\{0\},\; \ker G^*=\{0\}$. The case $\ker G^*\ne\{0\},\; \ker M=\{0\}$ one can be consider similarly.

Since $\ker G^*\ne\{0\}$ from \eqref{mar222a} it follows $\dom U_{*0}=\{0\}$ and from \eqref{juni27cc}
\[
\dom U_0=\ker M,\;\; U_0=D_{G^*}X\uphar\ker M.
\]
Since $\ker G^*=\{0\}$ implies $||D_{G^*}g||<||g||$ for all $g\in\sD_{G^*}\setminus\{0\}$ and $||Xh||\le ||h||$, we get \eqref{juni28e}.

\end{proof}

The theorem below is the inverse to Theorem \ref{mar11aaa} and Theorem \ref{juni27bb}.
\begin{theorem}\label{juni18a}
Let $S$ be a non-densely defined regular symmetric operator with non-zero semi-deficiency indices. Assume $\dim(\dom S)^\perp=\infty$.

Suppose that the dual pair $\left<\wt S_0,\wt S_{*0}\right>$ consists of closed extensions of $S_0$ in $\cH_0$ such that
\begin{itemize}
\item $\wt S_0$ and $-\wt S_{*0}$ are dissipative,
\item the equalities in \eqref{juli24a} hold.
\end{itemize}
Then there exist maximal dissipative extensions $\wt S$ of $S$ in $\cH$ such that
 $$P_{\cH_0}\wt S\uphar\dom \wt S\cap\cH_0= \wt S_0,\;
P_{\cH_0}\wt S^*\uphar\dom \wt S^*\cap\cH_0=\wt S_{*0}.$$
\end{theorem}

\begin{proof}
The domains of $\wt S_0$ and $\wt S_{*0}$ have the representations
\[
\dom \wt S_0=\dom S\dot +(I-U_0)\dom U_0,\; \dom \wt S_{*0}=\dom S\dot +(I-U_{*0})\dom U_{*0},
\]
 where $U_0$ and $U_{*0}$ are closed contractions
 \[\begin{array}{l}
\sN_{i}'\supseteq\dom U_0\stackrel{U_0}\longrightarrow\sN_{-i}',\; 
\sN_{-i}'\supseteq\dom U_{*0}\stackrel{U_{*0}}\longrightarrow\sN_{i}',\\
\end{array}
\]
satisfying conditions \eqref{juni28ff}. By Theorem \ref{juni27bb} the condition \eqref{juli24a} is equivalent to \eqref{juni27e}.

We will construct a contraction $\wt U\in\bB(\sN_i,\sN_{-i})$ of the form \eqref{mar11bbb}--\eqref{mar11bbbb}
such that (see Theorem \ref{mar11aaa} (2))
\[
\begin{array}{l}
\ker (I-Y)=\{0\},\;\ran (I-Y)\cap\ran D_{Y^*}=\{0\},\\
\dom U_0=\ker M,\; \dom U_{*0}=\ker G^*,\\
U_0=D_{G^*}X\uphar\ker M,\; U_{*0}=D_M X^*\uphar\ker G.
\end{array}
\]

Our construction is divided into several steps.

{\textbf{Step 1}}

(1) \textit{Constructions of auxiliary operators $V_0$, and $\cA_0$.}

Let $P_{*0}$ and $P^\perp_{*0}$ be the orthogonal projections in $\sN'_{-i}$ onto the subspaces $\dom U_{*0}$ and $\sN'_{-i}\ominus\dom U_{*0}$, respectively.
Then from condition \eqref{juni27e} it follows that
\[
||P^\perp_{*0} U_0\f_i||^2<||\f_{i}||^2-||P_{*0}U_0\f_i||^2
\;\;\forall \f_i\in\dom U_0: U_0\f_i\notin\dom U_{*0}.
\]
Set
$$\Phi_0:=P^\perp_{*0} U_0:\dom U_0\to\sN_{-i}'\ominus\dom U_{*0}$$
and let
\[
\Phi_0=|\Phi_0^*|V_0 
\]
be the polar decomposition of $\Phi_0$, where
\begin{itemize}
\item $\Phi_0^*:\sN_{-i}'\ominus\dom U_{*0}\to\dom U_{0}$ is the adjoint to $\Phi_0$,
\item $|\Phi_0^*|=(\Phi_0\Phi^*_0)^\half\in\bB(\sN_{-i}'\ominus\dom U_{*0})$ is the absolute value of $\Phi_0$, $\ran |\Phi_0^*|=\ran \Phi_0,$
\item $V_0:\dom U_{0}\to \sN_{-i}'\ominus\dom U_{*0}$
is a partial isometry
$$\ran V_0=\cran \Phi_0,\;
\;\ker V_0=\ker\Phi_0=\{h\in\dom U_0: U_0h\in\dom U_{*0}\}.$$
\end{itemize}
So, we have a representation
\begin{equation}\label{juni20a}
U_0=P_{*0}U_0+P_{*0}^\perp U_0=P_{*0}U_0+|\Phi^*_0| V_0.
\end{equation}
Let us show that
$$|||\Phi_0^*| h||<||h||\;\; \forall h\in\cran |\Phi_0^*|\setminus\{0\}=\cran \Phi_0\setminus\{0\}.$$

Assume $|||\Phi_0^*| h||^2=||h||^2$ for some $h\in\cran \Phi_0$. Then $h=V_0\f_i$ for some $\f_i\in\dom U_0,$
\[
\begin{array}{l}
\f_i=f_1+f_2,\; f_1\in\ker V_0,\;f_2\in\ran V^*_0=\dom U_0\ominus\ker V_0,\\[2mm]
h=V_0\f_i=V_0f_2,\\[2mm]
|||\Phi_0^*| h||^2=|||\Phi_0^*|V_0 \f_i||^2=|||\Phi_0^*|V_0 f_2||^2=||P^\perp_{*0} U_0 f_2||^2,\\[2mm]
||h||^2=||V_0\f_i||^2=||V_0f_2||^2=||f_2||^2.
\end{array}
\]
Hence
\[
||f_2||^2=||P^\perp_{*0} U_0 f_2||^2=||U_0f_2||^2- ||P_{*0} U_0 f_2||^2\le ||f_2||^2-||P_{*0} U_0 f_2||^2.
\]
Therefore $P_{*0} U_0 f_2=0$, i.e., $U_0f_2\notin \dom U_{*0}.$ Consequently by \eqref{juni27e}
\[
||f_2||^2>||U_0 f_2||^2=||P_{*0} U_0 f_2||^2+||P^\perp_{*0} U_0 f_2||^2.
\]
This inequality contradicts to the equality $||f_2||^2=||P^\perp_{*0} U_0 f_2||^2.$
It follows that $h=0.$

Define a nonnegative selfadjoint contraction $\cA_0\in\bB(\sN_{-i}')$ by means of the block operator matrix w.r.t the decomposition
$\sN_{-i}'=\dom U_{*0}\oplus(\sN_{-i}'\ominus\dom U_{*0})$:
\begin{equation}\label{juni24a}
\cA_0:=\begin{bmatrix}0&0\cr 0&I-|\Phi_0^*|^2  \end{bmatrix}:\begin{array}{l}\dom U_{*0}\\\oplus\\\sN_{-i}'\ominus\dom U_{*0}\end{array}\to \begin{array}{l}\dom U_{*0}\\\oplus\\\sN_{-i}'\ominus\dom U_{*0}\end{array}.
\end{equation}
Then $\ker \cA_0=\dom U_{*0}$. 

(2) \textit{Constructions of auxiliary operators $V_{*0}$, and $\cA_{*0}$.}

Let $P_{0}$ and $P^\perp_{0}$ be the orthogonal projections in $\sN'_{i}$ onto $\dom U_{0}$ and $\sN'_{i}\ominus\dom U_{0}$, respectively.

Set $\Phi_{*0}:=P^\perp_{0} U_{*0}:\dom U_{*0}\to\sN_{i}'\ominus\dom U_{0}$ and let
\[
\Phi_{*0}=|\Phi_{*0}^*|V_{*0} 
\]
be the polar decomposition of $\Phi_{*0}$, where
\begin{itemize}
\item
$\Phi_{*0}^*:\sN_{i}'\ominus\dom U_{0}\to\dom U_{*0}$ is the adjoint to $\Phi_{*0}$,
\item $|\Phi_{*0}^*|=(\Phi_{*0}\Phi^*_{*0})^\half\in\bB(\sN_{i}'\ominus\dom U_{0})$ is the absolute value of $\Phi_{*0}$, $\ran |\Phi_{*0}^*|=\ran \Phi_{*0},$
\item $V_{*0}:\dom U_{*0}\to \sN_{i}'\ominus\dom U_{0}$
is a partial isometry
$$\ker V_{*0}=\ker\Phi_{*0}=\{h\in\dom U_{*0}: U_{*0}h\in\dom U_{0}\}.$$
\end{itemize}
As above one can proof that
$$|||\Phi_{*0}^*\f_{-i}||<||\f_{-i}||\;\;
 \f_{-i}\in\cran |\Phi_{*0}^*|\setminus\{0\}=\cran \Phi_{*0}\setminus\{0\}.$$
We obtain the following representation of $U_{*0}$:
\begin{equation}\label{juni20b}
U_{*0}=P_{0}U_{*0}+P_{0}^\perp U_{*0}=P_{0}U_{*0}+|\Phi^*_{*0}| V_{*0}.
\end{equation}
Set
\begin{equation}\label{juni24b}
\cA_{*0}:=\begin{bmatrix}0&0\cr 0&I-|\Phi_{*0}^*|^2  \end{bmatrix}:\begin{array}{l}\dom U_{0}\\\oplus\\\sN_{i}'\ominus\dom U_{0}\end{array}\to \begin{array}{l}\dom U_{0}\\\oplus\\\sN_{i}'\ominus\dom U_{0}\end{array}.
\end{equation}

Then $\cA_{*0}$ is nonnegative selfadjoint contraction in $\sN_i'$ and  $\ker \cA_{*0}=\dom U_{*0}$. 

{\textbf{Step 2}}

(1) \textit{The choice of the operator $Y$}

By Remark \ref{juni22a} it is possible to find
a contraction $Y\in\bB(\sL_{-i})$ having the following properties:
\begin{equation}\label{aug8a}
\begin{array}{l}
({\bf{a}})\;\ker (I-Y)=\{0\},\;\ran (I-Y)\cap\ran D_{Y^*}=\{0\},\\
({{\bf b}})\;\dim \sD_{Y}\ge \dim(\sN_{-i}'\ominus\dom U_{*0}),\\
({{\bf c}})\;\dim \sD_{Y^*}\ge \dim(\sN_{i}'\ominus\dom U_{0}).
\end{array}
\end{equation}

(2) \textit{The operator $M$}

Let $\cA_{*0}$ be given by \eqref{juni24b}, let $Z_{*0}\in\bB(\cran A_{*0},\sD_{Y^*})$ be an isometry. Define a contraction
$$M:=Z_{*0}\cA_{*0}^\half \in\bB(\sN'_i,\sD_{Y^*}).$$
Then $\ker M=\dom U_0$, $\cran M^*=\sN_i'\ominus\dom U_0,$ and
\begin{equation}\label{juni20cc}
\begin{array}{l}
D^2_M=I-\cA_{*0}=\begin{bmatrix}I&0\cr 0&|\Phi_{*0}^*|^2  \end{bmatrix}:\begin{array}{l}\dom U_{0}\\\oplus\\\sN_{i}'\ominus\dom U_{0}\end{array}\to \begin{array}{l}\dom U_{0}\\\oplus\\\sN_{i}'\ominus\dom U_{0}\end{array},\\[3mm]
\sD_M^2=\dom U_0\oplus\cran \Phi_{*0}=\ker M\oplus \cran \Phi_{*0}.
\end{array}
\end{equation}
(3) \textit{The operator $G$}

Let $\cA_0$ be given by \eqref{juni24a}, let $Z_{0}\in \bB(\cran\cA_{0},\sD_{Y})$ be an isometry and
let the contraction $G\in\bB(\sD_Y,\sN_i')$ be defined via its adjoint $G^*\in\bB(\sN_i',\sD_{Y})$ as follows
$$G^*:=Z_{0}\cA_{0}^\half.$$
Then $G=\cA_{0}^\half Z_0^*$, $\ker G^*=\dom U_{*0}$, $\cran G=\sN_{-i}'\ominus\dom U_{*0}$ and
\begin{equation}\label{juni20dd}
\begin{array}{l}
D^2_{G^*}=I-\cA_{0}=\begin{bmatrix}I&0\cr 0&|\Phi_{0}^*|^2  \end{bmatrix}:\begin{array}{l}\dom U_{*0}\\\oplus\\\sN_{-i}'\ominus\dom U_{*0}\end{array}\to \begin{array}{l}\dom U_{*0}\\\sN_{-i}'\ominus\dom U_{*0}\end{array},\\[3mm]
\sD_{G^*}=\dom U_{*0}\oplus \cran\Phi_{0}=\ker G^*\oplus \cran\Phi_{0}.
\end{array}
\end{equation}

{\textbf{Step 3}}
\textit{Contractive operator $X\in\bB(\sD_M,\sD_{G^*})$}

Define contractions
\begin{equation}\label{jul20a}
\begin{array}{l}
X_{0}:\ker M=\dom U_0\to\sD_{G^*},\; X_0:=P_{*0}U_0+V_0,\\[2mm]
X_{*0}:\ker G^*=\dom U_{*0}\to\sD_{M},\; X_{*0}:=P_{0}U_{*0}+V_{*0}.
\end{array}
\end{equation}
Then \eqref{juni20a}, \eqref{juni20b}, \eqref{juni20cc}, \eqref{juni20dd}, \eqref{jul20a} imply the equalities:
\begin{equation}\label{juni20ee}
\begin{array}{l}
U_0=D_{G^*}X_0\uphar\ker M=(I-\cA_{0})^\half X_0\uphar \dom U_0=P_{*0}U_0+|\Phi_{0}^*|V_0,\\
 U_{*0}=D_M X_{*0}\uphar\ker G^*=(I-\cA_{*0})^\half X_{*0}\uphar \dom U_{*0}=P_{0}U_{*0}+|\Phi_{*0}^*|V_{*0}.
\end{array}
\end{equation}
Moreover, if $\f_i'\in\dom X_0$ and $\f_{-i}'\in\dom X_{*0}$, then since $D_{G^*}^{[-1]}\f_{-i}'=\f_{-i}'$, $D_M\f_i'=\f_i'$, we have
\[
\begin{array}{l}
(X_0 \f_i',\f_{-i}')=(D_{G^*}^{[-1]}U_0\f_i', \f_{-i}')=(U_0\f_i', D_{G^*}^{[-1]}\f_{-i}')=(U_0\f_i,\f_{-i}')=\\
=(\f_{i}', U_{*0}\f_{-i}')=(\f_i',D_M X_{*0}\f_{-i}')=(D_M\f_i', X_{*0}\f_{-i}')=(\f_i', X_{*0}\f_{-i}').
\end{array}
\]
Thus, if $\sH_1:=\sD_M$ and $\sH_2=:\sD_{G^*}$, then the pair $\left<X_0, X_{*0}\right>$ forms a dual pair of contractions acting from $\ker M=\dom X_0\subset\sH_1$ into $\sH_2$ and from $\ker G^*=\dom X_{*0}\subset\sH_2$ into $\sH_1$, respectively. According to \cite[Chapter IV, Proposition 4.2]{SF},
 \cite{AG1982} there exists contraction $X\in\bB(\sH_1,\sH_2)$ such that
\[
X\supset X_0,\;\; X^*\supset X_{*0}.
\]
This means that $X\uphar\ker M=X_0$, $X^*\uphar\ker G^*=X_{*0}.$

{\textbf{Step 4}}
\textit{Maximal dissipative extension $\wt S$}

Set
\[
\wt U:=
\begin{bmatrix} YV_i &D_{Y^*}M \cr
GD_{Y}V_i&-GY^*M+D_{G^*}XD_{M}\end{bmatrix}:\begin{array}{l}\sL_{i}\\ \oplus\\
\sN'_i\end{array}\to \begin{array}{l}\sL_{-i}\\ \oplus\\
\sN'_{-i}\end{array},
\]
where the operators $G$, $M$, $X$, and $Y$ have been described previously.
Define the maximal dissipative extension $\wt S$ of $S$ with
\[
\dom \wt S=\dom S\dot+(I-\wt U)\sN_i.
\]
Since $\ran (I-Y)\cap\ran D_{Y^*}=\ran (I-Y^*)\cap\ran D_{Y}=\{0\}$, from \eqref{may5a} by Theorem \ref{mar11aaa} (2) and from \eqref{juni20ee} we get
\[
\begin{array}{l}
\dom \wt U_0=\ker M=\dom U_0,\;\; \wt U_{0}=D_{G^*}X\uphar\ker M=U_0 \\
\dom \wt U_{*0}=\ker G^*=\dom U_{*0},\;\;\wt U_{*0}=D_{M}X^*\uphar\ker G=U_{*0}.
\end{array}
\]
According to Theorem \ref{mar11aaa}
the compressions of $\wt S$ and $\wt S^*$ on $\cH_0$ coincide with given extensions $\wt S_0$ and $\wt S_{*0}$ of $S_0$ in $\cH_0$.
\end{proof}
\begin{remark}\label{juni30a}
{\rm (1)} If $\wt S_{*0}=S_0$ (i.e., $\dom U_{*0}=\{0\}$), then the first condition in \eqref{juni27e} takes the form
\[
U_0h\ne 0\Longrightarrow ||U_0h||<||h|| 
\]
and implies that
\[
||U_0 h||<||h||\;\;\forall h\in\dom U_0\setminus\{0\}.
\]
Hence
$\Phi_0=U_0$, $U_0=|U_0^*|V_0 $
is the polar decomposition of $U_0$,
$$\begin{array}{l}
V_0:\dom U_0\to\sN_{-i}',\;
\ran V_0=\cran |U^*_0|=\cran U_0,\\
0\le |U_0^*|h<h$ for all $h\in\cran U_0^*,\\
\cA_0=I_{\sN_{-i}'}-|U_0^*|^2:\cran U_0\to\cran U_0,\; \ker \cA_0=\{0\},\\
 G^*:=Z_{0}\cA_{0}^\half =Z_0 (I_{\sN_{-i}'}-|U_0^*|^2)^\half,\; \ker G^*=\{0\},\;D_{G^*}=|U_0^*|\uphar\cran U_0,
\end{array} $$
where $Z_0$ is an isometry.
\[\begin{array}{l}
X_{0}:\ker M=\dom U_0\to\sD_{G^*}=\cran U_0,\; X_0:=V_0,\\[2mm]
X_{*0}=0:\ker G^*=\{0\}\to\sD_{M}. 
\end{array}
\]
Then $X$ is any contractive extension (see \cite{Crandall_1969}) of $X_0$ on the space $\sN_i'$.

In particular, if $\wt S_0=\wt S_{*0}=S_0$ ($\dom U_0=\{0\},\;\dom U_{*0}=\{0\}$), then $\wt S$ is determined by
\[
\wt U:=
\begin{bmatrix} YV_i &D_{Y^*}M \cr
GD_{Y}V_i&-GY^*M+D_{G^*}XD_{M}\end{bmatrix},
\]
where in \eqref{aug8a} ({{\bf a}}) remains the same, the conditions ({{\bf b}}), ({{\bf c}}) become $\dim \sD_{Y}\ge n_-'$ and $\dim \sD_{Y^*}\ge n_+'$, respectively, and
\[
\ker G^*=\{0\},\; \ker M=\{0\},\;
X\in\bB(\sD_M,\sD_{G^*})\;\mbox{is an arbitrary contraction}.
\]

{\rm(2)} If $\wt S_{*0}=(\wt S_0)^*$ in $\cH_0$ (i.e., $\dom U_0=\sN_i'$ and $\dom U_{*0}=\sN_{-i}'$), then $U_{*0}=(U_0)^*,$
by  \cite[Lemma 3.3]{ArlCAOT2023} and \cite[Theorem 4.5]{Arl_arxiv2024} the equalities \eqref{juli24a} are valid. This equivalent to
 $$||U_0h||\le ||h||\;\;\forall h\in\sN_i'.$$
Our construction of $\wt U$ gives
$
\wt U=\begin{bmatrix} YV_i &0 \cr
0&U_0\end{bmatrix}:\begin{array}{l}\sL_{i}\\ \oplus\\
\sN'_i\end{array}\to \begin{array}{l}\sL_{-i}\\ \oplus\\
\sN'_{-i}\end{array}
$
with an arbitrary contraction $Y\in\bB(\sL_{-i})$ such that $\ker (I-Y)=\{0\}.$
\end{remark}
\section{Compressions of exit space extensions of densely defined closed symmetric operators}
The following assertion is an application of Theorem \ref{bef29ab} and Theorem \ref{mar11aaa} to exit space extensions of a densely defined closed symmetric operators.
\begin{theorem}\label{aug06a} cf. 
Let $\cS$ be a densely defined closed symmetric operator in the Hilbert space $\cH$.  Assume that $\cH$ is a subspace of a Hilbert space $\sH$, $\cH_1:=\sH\ominus\cH,$
$ \dim\cH_1=\infty$.
If $n_{\pm}=\infty$, then in $\sH$ there exist
\begin{enumerate}
\item selfadjoint extensions $\wt \cS$ of $\cS$ such that
\[
\dom \wt \cS\cap\cH=\dom\cS \;\mbox{\rm{and}}\; \dom \wt \cS\cap\cH_1=\{0\}.
\]
\item maximal dissipative extensions $\wt S$ such that
\[
\dom \wt \cS\cap\cH=\dom\wt \cS^*\cap \cH=\dom\cS\; \mbox{\rm{and}}\; \dom \wt \cS\cap\cH_1=\dom\wt \cS^*\cap\cH_1=\{0\}.
\]
\end{enumerate}
\end{theorem}
\begin{proof} The operator $\cS$ is a non-densely defined and regular symmetric in $\sH$, the deficiency subspace $\cN_\lambda$ of $\cS$ in $\cH$ becomes the semi-deficiency subspace of $\cS$ in $\sH$, $\sH\ominus\ran (\cS-\bar\lambda I)=\cN_\lambda\oplus\cH_1$. For the subspaces $\sL_\lambda$ and the isometric operators $V_\lambda$ defined in \eqref{gjlghjcn} and \eqref{forbis} we have the equalities $\sL_\lambda= \cH_1$ and $V_\lambda=I_{\cH_1}$ for all $\lambda\in \dC\setminus\dR$.

Take a contraction $\cY$ in $\cH_1$ such that (see Remark \ref{juni22a}):
\[
\ker D_{\cY}=\{0\},\;\; \ran (I-\cY)\cap \ran D_{\cY^*}=\{0\}.
\]
Then $\sD_{\cY}=\sD_{\cY^*}=\cH_1.$

(1) Let 
\begin{equation}\label{aug8f}
\wt\cU=\begin{bmatrix} \cY &D_{\cY^*}\cM \cr
\cG D_{\cY}&-\cG\cY^*\cM\end{bmatrix}: \begin{array}{l}\\\cH_1\\\oplus\\
\cN_i\end{array}\to\begin{array}{l}\cH_1\\\oplus\\\cN_{-i}\end{array}
\end{equation}
be a unitary operator with entries having the properties:
\begin{itemize}
\item $\cG$ is a unitary operator from $\cH_1$ onto $\cN_i$,
\item $\cM$ is a unitary operator from $\cN_{-i}$ onto $\cH_1$.
\end{itemize}
Define a selfadjoint extension $\wt \cS$ in $\sH=\cH\oplus\cH_1$:
\[\begin{array}{l}
\dom \wt\cS=\dom \cS\dot+(I-\wt\cU)(\cH_1\oplus\cN_i).
\end{array}
\]
By Theorem \ref{bef29ab} we have $\dom\wt \cS\cap\cH=\dom \cS$. If $g\in\dom \wt S\cap\cH_1$, then
$$P_\cH g=f_{\cS}+f_i-\cG D_{\cY} f_1+\cG\cY^*\cM f_i =0,\; f_\cS\in\dom\cS, f_1\in\cH_1, f_i\in\cN_i.$$
 Because $\dom \cS^*=\dom \cS\dot+\cN_i\dot+\cN_{-i}$
we get
\[
f_\cS=0,\; f_i=0,\; \cG D_{\cY}f_1-\cG\cY^*\cM f_i=0
\]
Due to properties of $\cG$ and $D_{\cY}$ we conclude that $f_1=0$ and, therefore, $g=0$.

(2) Let us take non-unitary contractions $\cM$ and $\cG$ such that
$$\ker \cM=\ker \cM^*=\ker \cG=\ker \cG^*=\{0\}.$$
Let $\wt \cU$ be a contractive matrix of the form \eqref{aug8f} and let $\wt\cS$ be the corresponding maximal dissipative extension of $\cS$ in $\sH.$
Then from items (3) and (4) of Theorem \ref{mar11aaa} we conclude that $\dom \wt \cS\cap\cH=\dom\wt \cS^*\cap \cH=\dom\cS$.

Arguing further as above in (1), we arrive at the equalities
$$\dom \wt \cS\cap\cH_1=\dom\wt \cS^*\cap\cH_1=\{0\}.$$

\end{proof}

\begin{remark}\label{gua29a}
Statement (1) of Theorem \ref{aug06a} was estabished by Naimark in \cite[Corollary 13]{Naimark3}.
\end{remark}
\section{Compressions of the Shtraus extensions and characteristic functions of symmetric operators}

Recall that the Shtraus extension $\wt S_z$ of a symmetric operator $S$ takes the form \eqref{slamb}.

\begin{proposition}\label{htuek1}
Let $S$ be a regular symmetric operator and let $z\in\wh\rho(S)$. Then
\begin{enumerate}
\item the operator $P_{\cH_0}\wt S_z$ is closed;
\item the compression $P_{\cH_0}\wt S_z\uphar(\dom \wt S_z\cap\cH_0)$ coincides with the Shtraus extension $(\wt {S_0})_z$ of $S_0$;
\item if $n_+'=0$ and $n_-'>0$, then $\dom \wt S_z\subset\dom \wt S^*_z=\dom S^*;$
\item if $n_+'>0$ and $n_-'=0$, then $\dom \wt S^*_z\subset\dom \wt S_z=\dom S^*$;
\item if $S_0$ is selfadjoint in $\cH_0$, then
$\dom \wt S_z=\dom \wt S^*_z=\dom S^*$, the operator $(\wt S_z+\wt S^*_z)/2$ is selfadjoint and the operator $(\wt S_z-\wt S^*_z)/2i$ is bounded and essentially selfadjoint.
\end{enumerate}
\end{proposition}
\begin{proof} (1) For $f_S\in\dom S$ and $\f_z\in \sN_z$ we have
\[
P_{\cH_0}\wt S_z(f_S+\f_z)-z(f_S+\f_z)=(S_0-z I)f_S-zP_{\sL}\f_z
\]
Hence, if
\[
\lim\limits_{n\to\infty}(f_S^{(n)}+\f_z^{(n)})=f,\; \lim\limits_{n\to\infty}P_{\cH_0}\wt S_z(f_S^{(n)}+\f_z^{(n)})=h,
\]
then the sequence $\{(S_0-z I)f_S^{(n)}\}$ converges. This implies that the sequence $\{f_S^{(n)}\}$ converges and therefore, the sequences
$
\{S_0f_S^{(n)}\}$ and $ \{\f_z^{(n)}\}$ converge. Because the operator $S_0$ is closed and $\sN_z$ is a subspace we conclude that the vector
$g:=\lim\limits_{n\to \infty}f_S^{(n)}$ belongs to $\dom S_0(=\dom S)$, the vector  $\f_z:=\lim\limits_{n\to \infty}\f_z^{(n)}$ belongs to $\sN_z$,
and
$$
f=\lim\limits_{n\to\infty}(f_S^{(n)}+\f_z^{(n)})=g+\f_z\in\dom \wt S_z,\;
h= \lim\limits_{n\to\infty}P_{\cH_0}\wt S_z(f_S^{(n)}+\f_z^{(n)})= P_{\cH_0}\wt S_z(g+\f_z),$$
 i.e., the operator $P_{\cH_0}\wt S_z$ is closed.

(2) Clearly
$$(\dom S\dot+\sN_z)\cap\cH_0=\dom S_0\dot+(\sN_z\cap\cH_0)=\dom S_0\dot+\sN_z'.$$
Hence
$\dom \wt S_z\cap\cH_0=\dom (\wt {S_0})_z$ and $(\wt {S_0})_z=P_{\cH_0}\wt S_z\uphar\cH_0.$

(3) and (4) are consequences of the equality (see\eqref{rhfc11})
$$ \dom\wt S_z\cap\dom \wt S^*_z=\dom S\dot+\sL_z=\dom S\dot+\sL_{\bar z}.$$

(5) Because $S_0$ is selfadjoint in $\cH_0$ and the operators $P_{\cH_0}\wt S_z$ and $P_{\cH_0}\wt S^*_z$ are closed, by \cite[Theorem 4.1.12]{ArlBelTsek2011}
the statements in (5) hold.
\end{proof}

In the next proposition we establish connections between characteristic functions of symmetric operators $S$ and $S_0$.
\begin{proposition}\label{apr10bb}
Let $S$ be a non-densely defined regular symmetric operator with non-zero semi-deficiency indices, let $C_{\lambda}^S(z)\in\bB(\sN_{\lambda},\sN_{{\bar\lambda}})$ and
$C_{\lambda}^{S_0}(z)\in\bB(\sN_{\lambda}',\sN_{{\bar\lambda}}')$ be the characteristic functions of $S$ and $S_0$, respectively, $\lambda, z\in\dC_+$ (see Section \ref{jul11}).

Represent the function $C_{\lambda}^S(z)$ in the block-matrix form
\[
C_{\lambda}^S(z)=\begin{bmatrix}A^{(\lambda)}_{11}(z)&A^{(\lambda)}_{12}(z)\cr A^{(\lambda)}_{21}(z)&A^{(\lambda)}_{22}(z)  \end{bmatrix} :\begin{array}{l}\sL_\lambda\\\oplus\\\sN_\lambda'\end{array}\to \begin{array}{l}\sL_{\bar\lambda}\\\oplus\\\sN_{\bar\lambda}'\end{array},
\]
where
\[
\begin{array}{l}
A^{(\lambda)}_{11}(z)= P_{\sL_{\bar\lambda}}C_{\lambda}^S(z)\uphar\sL_\lambda,\; A^{(\lambda)}_{12}(z)= P_{\sL_{\bar\lambda}}C_{\lambda}^S(z)\uphar\sN'_\lambda,\\[2mm]
A^{(\lambda)}_{21}(z)= P_{\sN'_{\bar\lambda}}C_{\lambda}^S(z)\uphar\sL_\lambda,\; A^{(\lambda)}_{22}(z)= P_{\sN_{\bar\lambda}'}C_{\lambda}^S(z)\uphar\sN'_\lambda,\;
\lambda\in\dC_+.
\end{array}
\]
Then  $C_{\lambda}^{S_0}(z)$ is the Schur complement of the block operator-matrix function
$$C_{\lambda}^S(z)-V_{\lambda}P_{\sL_{\lambda}}=\begin{bmatrix}A^{(\lambda)}_{11}(z)-V_{\lambda}&A^{(\lambda)}_{12}(z)\cr A^{(\lambda)}_{21}(z)&A^{(\lambda)}_{22}(z)  \end{bmatrix} :\begin{array}{l}\sL_\lambda\\\oplus\\\sN_\lambda'\end{array}\to \begin{array}{l}\sL_{\bar\lambda}\\\oplus\\\sN_{\bar\lambda}'\end{array},$$
i.e., the following relations for any number $z\in\dC_+$ are valid:
\begin{equation}\label{apr15a}
\begin{array}{l}
C_{\lambda}^{S_0}(z)=A^{(\lambda)}_{22}(z)-A^{(\lambda)}_{21}(z)(A^{(\lambda)}_{11}(z)-V_{\lambda})^{-1}A^{(\lambda)}_{12}\\
=\biggl(P_{\sN_{\bar\lambda}'}C_{\lambda}^S(z)-
P_{\sN'_{\bar\lambda}}C_{\lambda}^S(z)(P_{\sL_{\bar\lambda}}C_{\lambda}^S(z)-V_{\lambda})^{-1}P_{\sL_{\bar\lambda}}C_{\lambda}^S(z)\biggr)\uphar\sN'_\lambda,
\end{array}
\end{equation}
\begin{multline}\label{apr15b}
\qquad \qquad C_{\lambda}^S(z)-V_{\lambda}P_{\sL_{\lambda}}=\\[2mm]
\qquad=\begin{bmatrix}I_{\sL_{\bar\lambda}}&0\cr A^{(\lambda)}_{21}(z)(A^{(\lambda)}_{11}(z)-V_{\lambda})^{-1}&I_{\sN_{\bar\lambda}'}\end{bmatrix}\begin{bmatrix} A^{(\lambda)}_{11}(z)-V_{\lambda}&0\cr 0&C_{\lambda}^{S_0}(z)\end{bmatrix}\begin{bmatrix}I_{\sL_{\lambda}}&(A^{(\lambda)}_{11}(z)-V_{\lambda})^{-1}A^{(\lambda)}_{12}(z)\cr 0&I_{\sN_\lambda'}  \end{bmatrix}. 
\end{multline}
\end{proposition}
\begin{proof}

 Recall that by Proposition \ref{htuek1}  the operator $\wt S_z$ is a regular extension of $S$. Therefore, the linear manifold $(C^{(S)}_{\lambda}(z)-V_\lambda)\sL_\lambda$ is a subspace (see Remark \ref{mar9bb}).
Besides, the relations \eqref{apr9bb} and \eqref{31jul} for $C^{S}_{\lambda}(\cdot)$ hold.
It follows that the operator $(P_{\sL_\lambda}C^{S}_{\lambda}(z)-V_\lambda)\uphar\sL_\lambda=A^{(\lambda)}_{11}(z)-V_{\lambda}$ has bounded inverse defined on $\sL_\lambda$ for all $z\in\dC_+$.

Since the operator $\wt S_z$ is maximal dissipative extension of $S$,
\[
\dom\wt S_z=\dom S\dot+(I-C^{S}_{\lambda}(z))\sN_\lambda,
\]
and $P_{\cH_0}\wt S_z\uphar(\dom\wt S_z\cap\cH_0)=(\wt {S_0})_z$ (by Proposition \ref{htuek1}),
\[
\dom (\wt {S_0})_z=\dom S_0\dot+\sN_z'=\dom S_0\dot +(I-C_{\lambda}^{S_0}(z))\sN_\lambda',
\]
in order to get \eqref{apr15a} and \eqref{apr15b} we can apply Proposition \ref{mar11a}, Theorem \ref{mar11aaa} (statements (2) and (8)).
\end{proof}

\end{document}